\documentclass[reqno,oneside,11pt,letterpaper]{amsart}
\usepackage{dwpreamble}
\usepackage{dwcommands}

\renewcommand{\ph}{{\rm cl}}

\begin{document}

\title[Dynamical residues of Lorentzian spectral zeta functions]{\Large{Dynamical residues  of Lorentzian spectral zeta functions}\smallskip}

\author{}
\address{Institut de Mathématiques de Jussieu, Sorbonne Université -- Université de Paris, 4 pl.~Jussieu,
75252 Paris, France}
\email{nguyen-viet.dang@imj-prg.fr}
\author[]{\normalsize Nguyen Viet \textsc{Dang} \& Micha{\l} \textsc{Wrochna}}
\address{Laboratoire Analyse G\'eom\'etrie Mod\'elisation, CY Cergy Paris Universit\'e, 2 av.~Adolphe Chauvin, 95302 Cergy-Pontoise, France}
\email{michal.wrochna@cyu.fr}

\begin{abstract} We define a dynamical residue which  generalizes the Guillemin--Wodzicki residue density of pseudo-differential operators. More precisely, given a Schwartz kernel,  the definition refers to  Pollicott--Ruelle resonances for the dynamics of scaling towards the diagonal. We apply this formalism to complex powers of the wave operator and we prove that residues of Lorentzian spectral zeta functions are dynamical residues. The residues are shown to have local geometric content as expected from formal analogies with  the Riemannian case.
\end{abstract}

\maketitle

\section{Introduction}

\subsection{Introduction and main results} Suppose $(M,g)$ is a compact Riemannian manifold of dimension $n$, and let $\triangle_g$ be the Laplace--Beltrami operator. A classical result in analysis, dating back to  Minakshisundaram--Pleijel \cite{Minakshisundaram1949} and Seeley \cite{seeley}, states that the trace density of $(-\triangle_g)^{-\cv}$ is well-defined for $\Re \cv > \n2$ and extends to a density-valued meromorphic function of the complex variable $\cv$. The meromorphic continuation, henceforth denoted  by $\zeta_g(\cv)$,  gives after integrating on $M$ the celebrated  \emph{spectral zeta function} of $-\triangle_g$ (or \emph{Minakshisundaram--Pleijel zeta function}). 

A fundamental fact shown independently by Wodzicki \cite{wodzicki} and Guillemin \cite{Guillemin1985} is that  each residue of $\zeta_g(\cv)$ equals an integral of a distinguished term in the polyhomogeneous expansion of the symbol of $(-\triangle_g)^{-\cv}$. The so-defined  \emph{Guillemin--Wodzicki residue density}  is remarkable because it has an {intrinsic} meaning and involves local geometric quantities, such as the  scalar curvature $R_g$ for even $n\geqslant 4$. It can also be intrinsically defined for more general classes of elliptic pseudo-differential operators (see \sec{ss:wodzicki}) and has a deep relationship with the Dixmier trace found by Connes \cite{Connes1988} (cf.~Connes--Moscovici \cite{Connes1995}).      

If now $(M,g)$ is a Lorentzian manifold (not necessarily compact), the corresponding Laplace--Beltrami operator $\square_g$, better known as the wave operator or d'Alembertian,  is far from being elliptic. However, it was recently shown that if $(M,g)$ is well-behaved at infinity or has special symmetries,  $\square_g$ is essentially self-adjoint in $L^2(M,g)$ \cite{derezinski,vasyessential,nakamurataira,Derezinski2019}, and consequently complex powers $(\square_g-i \varepsilon)^{-\cv}$ can be defined by functional calculus for any $\varepsilon>0$. Furthermore, for large $\Re \cv$,  $(\square_g-i \varepsilon)^{-\cv}$ has a well-defined trace-density, which extends to a meromorphic function \cite{Dang2020}, denoted from now on by $\zeta_{g,\varepsilon}(\cv)$. The residues of the so-obtained \emph{Lorentzian spectral zeta function density}  $\zeta_{g,\varepsilon}(\cv)$ contain interesting geometric information (for instance the Lorentzian scalar curvature $R_g$ occurs in the residue at $\cv=\frac{n}{2}-1$ for even $n\geqslant4$ \cite{Dang2020}), so it is natural to ask if these analytic residues coincide with a suitable generalization of the Guillemin--Wodzicki  residue.

The problem is that the notion of Guillemin--Wodzicki residue relies on the symbolic calculus of pseudo\-differential operators, and even though there is a natural generalization to Fourier integral operators due to Guillemin \cite{Guillemin1993} (see also \cite{Hartung2015}), Lorentzian complex powers fall outside of that class in view of their on-diagonal behavior. A priori one needs therefore a more singular calculus, based for instance  on paired Lagrangian distributions \cite{Guillemin1981,Melrose1979,Antoniano1985,Greenleaf1990,joshi,Joshi1998}. 

Instead of basing the analysis on a detailed symbolic calculus, the idea pursued in the present paper (and implicit in the work of Connes--Moscovici \cite{Connes1995}) is that regardless of how the calculus is obtained,  terms of different order should be distinguished by different scaling behavior as one approaches the diagonal $\Delta\subset M\times M$ of the Schwartz kernel.  We define the scaling as being generated by an \emph{Euler vector field} $\euler$ (see \sec{ss:euler}), the prime example being $X=\sum_{i=1}^n h^i \p_{h^i}$ if $(x,h)$  are local coordinates in which the diagonal is $\Delta=\{  h^{i}=0, \ i=1,\dots,n \}$. Now if $u$ is a distribution defined near $\Delta\subset M\times M$ and it scales in a log-polyhomogeneous way, the Laplace transform
\beq\label{eq:lapl}
s\mapsto \int_0^\infty e^{-ts} { \left(e^{-t\euler}u\right)} \,dt
\eeq
is a meromorphic function with values in distributions, and the poles are called  \emph{Pollicott--Ruelle resonances} \cite{Pollicott1986,Ruelle1986}. We define the \emph{dynamical residue} $\res_X u$ as the trace density of  $X \Pi_0(u)$ where $\Pi_0(u)$ is the  residue at $s=0$ of \eqref{eq:lapl}. 

As a first consistency check, we show that the dynamical residue and the  Guillemin--Wodzicki residue coincide for classical pseudo\-differential operators (i.e., with one-step polyhomogeneous symbol).

\begin{theorem}[{cf.~Theorem \ref{wodzickipdo}}]\label{tthm1} For any classical $A\in \Psi^m(M)$ with Schwartz kernel $K_{A}$, the dynamical residue {$\res_X \pazocal{K}_{A}$}  is well-defined, independent on the choice of Euler vector field $X$, and {$(\res_X \pazocal{K}_{A}) \dvol_g$} equals the Guillemin--Wodzicki residue density of $A$.  
\end{theorem}

Next, we consider  the case of a Lorentzian manifold $(M,g)$ of even dimension $n$. 

The well-definiteness  and  meromorphic continuation of $\zeta_{g,\varepsilon}(\cv)$ is proved in \cite{Dang2020} in the setting of globally hyperbolic \emph{non-trapping Lorentzian scattering spaces} introduced by Vasy \cite{vasyessential}. This class is general enough to contain perturbations of Minkowski space, one can however expect that it is not the most general possible for which $\zeta_{g,\varepsilon}(\cv)$ exists. For this reason, instead of making assumptions on $(M,g)$ directly, we point out the analytic properties which guarantee that $\zeta_{g,\varepsilon}(\cv)$ is a well-defined meromorphic function. Namely, we assume that $\square_g$ has \emph{Feynman resolvent}, by which we mean that:\smallskip
\ben
\item[]  $\square_g$ acting on $C_{\rm c}^\infty(M)$ has a self-adjoint extension, and the resolvent $(\square_g-z)^{-1}$ of this self-adjoint extension has \emph{Feynman wavefront set} uniformly in $\Im z >0$.
\een
\smallskip

\noindent The Feynman wavefront set condition roughly says that microlocally, the Schwartz kernel of $(\square_g-z)^{-1}$ has the same singularities as the \emph{Feynman propagator} on Minkowski space, i.e.~the Fourier multiplier by $(-\xi_0^2 + \xi_1^2 + \cdots+ \xi_{n-1}^2 - i0)^{-1}$ (see \cite{GHV,Vasy2017b,vasywrochna,GWfeynman,Gerard2019b,Taira2020a} for results    in this direction with fixed $z$). 
 The precise meaning of  uniformity  is given in Definition \ref{deff} and involves decay in $z$ along the integration contour used to define complex powers. We remark that outside of the class of Lorentzian scattering spaces, $\square_g$ is known to have Feynman resolvent for instance on ultra-static spacetimes with compact Cauchy surface, see Derezi\'nski--Siemssen \cite{derezinski} for the self-adjointness and \cite{Dang2020} for the microlocal estimates.


Our main result can be summarized as follows.

\begin{theorem}[{cf.~Theorem \ref{thm:dynres1}}]\label{thm:dynres} Let $(M,g)$ be a Lorentzian manifold of even dimension $n$, and suppose $\square_g$ has Feynman resolvent. For all $\cv\in\cc$ and $\Im z > 0$, the dynamical residue $\resdyn(\square_g-z)^{-\cv}$ is well-defined and independent on the choice of  Euler vector field $\euler$. Furthermore, for all $k=1,\dots,\frac{n}{2}$ and $\varepsilon>0$,
\beq\label{eq:main}
\resdyn \left(\square_g-i \varepsilon \right)^{-k} = {2}\res_{\cv =k}\zeta_{g,\varepsilon}(\cv),
\eeq
where $\zeta_{g,\varepsilon}(\cv)$ is the spectral zeta function density of $\square_g-i \varepsilon$. 
\end{theorem}


By Theorem \ref{tthm1},  the dynamical residue is a generalization of the Guillemin--Wodzicki residue density.  Thus, Theorem \ref{thm:dynres} { generalizes to the Lorentzian setting results known previously only in the elliptic case: the analytic poles of spectral zeta function densities  coincide with a more explicit quantity which refers to the scaling properties of complex powers. }   In physicists' terminology, this gives precise meaning to the statement that the residues of $\zeta_{g,\varepsilon}(\cv)$ can be interpreted as \emph{scaling anomalies}.

We also give a more direct expression for the l.h.s.~of \eqref{eq:main} which allows to make the relation with local geometric quantities, see \eqref{eq:explicit} in the main part of the text. In particular, we obtain in this way the identity (which  also follows from \eqref{eq:main} and \cite[Thm.~1.1]{Dang2020}) for $n\geqslant 4$:
\beq\label{eq:main2}
\lim_{\varepsilon\to 0^+}\resdyn \left(\square_g-i \varepsilon \right)^{-\frac{n}{2}+1}  = \frac{R_g(x)}{3i\Gamma(\frac{n}{2}-1)\left(4\pi\right)^{\frac{n}{2}}}.
\eeq
This identity implies that the l.h.s.~can be interpreted as a spectral action for gravity.

\subsection{Summary} The notion of dynamical residue is introduced in \sec{section2}, preceded by preliminary results on Euler vector fields. A pedagogical model is given in \sec{ss:toy} and serves as a motivation for the definition.  

  The equivalence of the two notions of residue for pseudo-differential operators (Theorem \ref{tthm1}) is proved in \sec{ss:wodzickipdo}.  An important role is played by the so-called Kuranishi trick which allows us to  adapt the phase of quantized symbols  to the coordinates in which a given Euler field $X$ has a particularly simple form. 

The remaining two sections \secs{section4}{section5} are devoted to the proof of Theorem \ref{thm:dynres}. 

The main ingredient is the \emph{Hadamard parametrix} $H_N(z)$ for $\square_g-z$, the construction of which we briefly recall in  \sec{s:hadamardformal}. Strictly speaking, in the Lorentzian case there are several choices of parametrices: the  one relevant here  is the \emph{Feynman Hadamard parametrix}, which approximates $(\square_g-z)^{-1}$ thanks to the Feynman  property combined with  uniform estimates for $H_N(z)$  shown in \cite{Dang2020}. The log-homogeneous expansion of the Hadamard parametrix $H_N(z)$ is shown in  \sec{ss:polhom} through  an oscillatory integral representation with singular symbols. An important role is played again by the Kuranishi trick adapted from the elliptic setting. { However, there are  extra difficulties due to the fact that we do not work with standard symbol classes anymore:  the ``symbols''  are distribution-valued and  special care is required when operating with expansions and controlling the remainders.} The dynamical residue is computed in \sec{ss:rcc}  with the help of extra expansions that exploit the homogeneity of individual terms and account for the dependence on $z$.

Next, following \cite{Dang2020} we introduce in \sec{ss:hc} a generalization  $H_N^{(\cv)}(z)$ of the Hadamard parametrix for complex powers $(\square_g-z)^{-\cv}$, and we adapt  the analysis from \sec{section4}. Together with the fact (discussed in \sec{ss:lg}) that $H_N^{(\cv)}(z)$ approximates $(\square_g-z)^{-\cv}$, this allows us to conclude the theorem. 

{ As an aside, in Appendix \ref{app} we briefly discuss what happens when $(\square_g-z)^{-\cv}$ is replaced by $Q(\square_g-z)^{-\cv}$ for an arbitrary differential operator $Q$.  We show that in this greater generality, the trace density still exists for large $\Re \cv$ and analytically continues  to at least $\cc\setminus \zz$.  This  can be interpreted as an analogue of the Kontsevich--Vishik canonical trace density \cite{Kontsevich1995}  in our  setting.  }

\subsection{Bibliographical remarks}\label{ss:br}  

Our approach to the Guillemin--Wodzicki residue  \cite{wodzicki,Guillemin1985}  is strongly influenced by  works in the  pseudodifferential setting by Connes--Moscovici \cite{Connes1995},  Kontsevich--Vishik \cite{Kontsevich1995}, Lesch \cite{Lesch}, Lesch--Pflaum \cite{Lesch2000}, Paycha \cite{paycha2,paycha} and Maeda--Manchon--Paycha \cite{Maeda2005}.

It  also draws from the theory of Pollicott--Ruelle resonances \cite{Pollicott1986,Ruelle1986} in the analysis and spectral theory 
 of hyperbolic dynamics (see Baladi \cite{Baladi2018} for a review of the subject and further references), in particular from the work of Dyatlov--Zworski \cite{Dyatlov2016} on dynamical zeta functions.

The Feynman wavefront set condition plays an important role in  various
 developments connecting the global theory of hyperbolic operators with
local geometry, in particular in works on index theory  by  B\"ar--Strohmaier and other authors \cite{Bar2019,Baer2020,Shen2021}, and on trace formulae and Weyl laws by Strohmaier--Zelditch \cite{Strohmaier2020b,Strohmaier2020,Strohmaier2020a} (including  a spectral-theoretical formula for the scalar curvature).

The Hadamard parametrix for inverses of the Laplace--Beltrami operator is a classical tool in analysis, see e.g.~\cite{HormanderIII,soggeHangzhou,StevenZelditch2012,Zelditch2017} for the Riemannian or Lorentzian time-independent case. For fixed $z$, the Feynman Hadamard parametrix is constructed by Zelditch \cite{StevenZelditch2012} in the ultra-static case and in the general case by Lewandowski \cite{Lewandowski2020}, cf.~Bär--Strohmaier \cite{Baer2020} for a unified treatment of even and odd dimensions. The present work  relies on the construction and the uniform in  $z$ estimates from \cite{Dang2020}, see also Sogge \cite{Sogge1988}, Dos Santos Ferreira--Kenig--Salo \cite{Ferreira2014}  and Bourgain--Shao--Sogge--Yao \cite{Bourgain2015} for uniform estimates in the Riemmanian case.

 In Quantum Field Theory on Lorentzian manifolds, the Hadamard parametrix plays a fundamental role in  renormalization, see e.g.~\cite{DeWitt1975,Fulling1989,Kay1991,Radzikowski1996,Moretti1999,Brunetti2000,Hollands2001}. Other rigorous renormalization schemes (originated in  works by Dowker--Critchley \cite{Dowker1976} and Hawking \cite{Hawking1977}) use a formal, \emph{local} spectral zeta function or heat kernel, and their relationships with the Hadamard parametrix were studied by Wald \cite{Wald1979}, Moretti \cite{Moretti1999,morettilong} and Hack--Moretti \cite{Hack2012a}. We remark   in this context that in Theorem \ref{thm:dynres} we can replace  globally defined complex powers $(\square_g-z)^{-\cv}$ with the local parametrix $H_N^{(\cv)}(z)$ and correspondingly we can replace the spectral zeta density $\zeta_{g,\varepsilon}(\cv)$ by a local analogue $\zeta_{g,\varepsilon}^{\rm loc}(\cv)$ defined using $H_N^{(\cv)}(z)$. This weaker, local formulation does not use the Feynman condition and thus holds true generally.

\subsection*{Acknowledgments} { We thank the anonymous reviewers for their useful suggestions and feedback. } Support from the grant ANR-16-CE40-0012-01 is gratefully acknowledged. The authors  also grateful to the MSRI in Berkeley and the Mittag--Leffler Institute in Djursholm for their kind hospitality during thematic programs and workshops in 2019--20.   


\section{Log-polyhomogeneous scaling and dynamical residue}\label{section2}

\subsection{Notation} { Throughout the paper, given a vector field $V\in C^\infty(T\pazocal{M})$ and a smooth function $f\in C^\infty(\pazocal{M})$ on a smooth manifold $\pazocal{M}$, we  denote by $e^{tV}:\pazocal{M}\mapsto \pazocal{M}$, $t\in \mathbb{R}$ the  flow generated by $V$, and by $e^{-tV}f:=f(e^{-tV}.)\in C^\infty(\pazocal{M})$ the pull-back of $f$ by the flow $e^{-tV}$. Furthermore, when writing  $Vf\in C^\infty(\pazocal{M})$ we will mean that  the vector field $V$ acts on $f$ by {Lie derivative}, i.e.~$Vf=\left(\frac{d}{dt}\left(e^{tV}f\right)\right)|_{t=0}$. }

\subsection{Euler vector fields and scaling dynamics}\label{ss:euler}

Let $M$ be a smooth manifold, and let $\Delta=\{(x,x) \st x\in M\}$ be the diagonal in $M\times M$. Our first objective is to  introduce a class of 
{ 
{Schwartz kernels} defined in some neighborhood of $\Delta$},
which have prescribed 
analytical 
behavior under scaling with respect to $\Delta$. 

More precisely, an adequate notion of scaling is provided by the dynamics generated by the following class of vector fields.

\begin{defi}[Euler vector fields]
Let  $\pazocal{I}\subset C^\infty(M\times M)$  be the ideal of smooth functions 
vanishing at the diagonal $\Delta=\{(x,x) \st x\in M\}\subset M\times M$ and $\pazocal{I}^k$ its $k$-th power.
{ A vector field} $\euler$ 
defined near the diagonal 
$\Delta$
is called \emph{Euler} if near $\Delta$, $\euler f=f+\pazocal{I}^2$ for all $f\in \pazocal{I}$.

For the sake of simplicity,  we will only consider { Euler vector fields $X$ scaling with respect to the diagonal} which in addition preserve the fibration 
$\pi: M\times M \ni (x,y)\mapsto x\in M$ 
 projecting on the first factor. We refer to any such $\euler$ simply as to an \emph{Euler vector field}. 
\end{defi}

{ In our definition, $\euler$ only needs to be defined on some neighborhood of $\Delta$ which is stable by the dynamics.
Euler vector fields appear to have been first defined { by Mark Joshi, who called them \emph{radial vector fields}. They were used in his works \cite{10.2307/2162103,Joshi1998} for defining polyhomogeneous Lagrangian and paired Lagrangian distributions by scaling}. Then unaware of Joshi's work, it appeared in the first author's thesis \cite{dangthesis}, see also \cite[Def.~1.1]{DangAHP}. They were independently found by Bursztyn--Lima--Meinrenken \cite{Bursztyn2019}, see also  \cite{Bischo2020} and the survey \cite{Meinrenken2021}. 
}

A consequence of the definition of Euler vector fields $\euler$ is that
if $f\in \pazocal{I}^k$ then $\euler f-kf\in \pazocal{I}^{k+1}$ which is  
easily proved by induction using Hadamard's lemma.
Another useful consequence of the definition of $\euler$ is that
we have the equation:
\begin{equation}\label{e:idneardiageuler}
{ \left(Xdf-df\right)|_{\Delta}=0}
\end{equation}
for all { smooth functions $f$ defined near} $\Delta$, { where $\euler df$ means the vector field $\euler$ acting on the $1$-form $df$ by Lie derivative, and $|_\Delta$ means the restriction on the diagonal.}
{ The equation (\ref{e:idneardiageuler})} can be easily checked by an immediate coordinate calculation.
We view $df|_\Delta$ as a smooth section of $T^*M^2$, a $1$--form, restricted over $\Delta$.

{ Recall that for $t\in \rr$, $e^{t\euler}$  is the flow of $\euler$ at time $t$. }

\begin{ex}
On $\mathbb{R}^4$, the dynamics $ e^{t\euler}:  \left(\mathbb{R}^4\right)^2\ni(x,y) \mapsto (x,e^{t}(y-x)+x)\in \left(\mathbb{R}^4\right)^2 $ preserves
the fibers of $\left(\mathbb{R}^4\right)^2 \ni (x,y)\mapsto x\in \mathbb{R}^4$.
\end{ex}

Euler vector fields can be obtained from any torsion-free connection $\nabla$ and the geodesic exponential  $\exp_x^\nabla:T_xM\to M$ defined using $\nabla$.
Namely, a \emph{geodesic Euler vector field} is obtained by setting
$$
{ \euler f(x,y)=\frac{d}{dt}f\big(x,\exp^\nabla_x(tv)\big)|_{t=1}},
$$
where $y=\exp_x^\nabla(v)$. Moreover, Euler vector fields form a \emph{particular class} of the Morse--Bott vector fields
where $\Delta$ is the critical manifold, the Morse index is $0$ and all Lyapunov exponents of $\euler$ equal $1$ or $0$.

Let us describe in simple terms the dynamics of Euler vector fields. 
\begin{lemm}[{Lyapunov exponents and bundles}]
Let $\euler$ be an Euler vector field. There exists a unique subbundle
$N\Delta\subset T_\Delta \left(M\times M\right)$ such that
$de^{t\euler}=e^{t}\,\id: N\Delta\to N\Delta$~\footnote{In the terminology of dynamical systems, this is a simple instance of a Lyapunov
bundle.}. 
\end{lemm}

\begin{proof}
The flow $e^{-t\euler}$ fixes $\Delta$ 
hence the differential $de^{-t\euler} :TM^2\to TM^2$ restricted to $\Delta$ defines a family of bundle isomorphisms
$de^{-t\euler}: T M^2|_\Delta \to T M^2|_\Delta$, $\forall t\in \mathbb{R}$.
Now using the group property of the flow $e^{-t\euler}e^{-s\euler}=e^{-(t+s)\euler}$,
we deduce that $de^{-tX}de^{-sX}=de^{-(t+s)\euler}:T M^2|_\Delta \to T M^2|_\Delta$.
We define the bundle map $L_\euler:T M^2|_\Delta \to T M^2|_\Delta$ as
$\frac{d}{dt}de^{t\euler}|_{t=0}$, which is the linearized action of $\euler$ localized at $\Delta$. By uniqueness of solutions to ODE and the group property of $de^{t\euler}:T M^2|_\Delta \to T M^2|_\Delta$, we find that $de^{t\euler}=e^{tL_\euler}:T M^2|_\Delta \to T M^2|_\Delta, \forall t\in \mathbb{R}$. Recall that for all smooth germs $f$ near $\Delta$, we have 
$Xdf=df|_\Delta$, we view $df|_\Delta$ as a smooth section of $T^*M^2$ over $\Delta$. Now we observe the following { identity} on $1$-forms restricted over $\Delta$: 
{
\begin{eqnarray*}
\forall f, df=
\euler df= \Big(\frac{d}{dt}
\left(e^{t\euler}df\right)\Big)|_{t=0}=\Big(\frac{d}{dt}
d\left(e^{t\euler}f\right)\Big)|_{t=0}=\Big(\frac{d}{dt}
\left(df\circ de^{t\euler}\right)\Big)|_{t=0}=L_\euler^* df 
\end{eqnarray*}
}
where $L_\euler^*:T^*M^2|_\Delta\to T^*M^2|_\Delta$ is the transpose of {$L_\euler$}.
The above equation implies that the eigenvalues of the bundle map $L_\euler:TM^2|_\Delta\mapsto TM^2|_\Delta$ are {$1$ or $0$}. So we define $N\Delta\subset TM^2|_\Delta$ as the eigenbundle of $L_\euler$ for the eigenvalue $1$.
\end{proof}

\begin{lemm}[Stable neighborhood]\label{l:stable}
There  exists a neighborhood $\pazocal{U}$ of $\Delta$ in $M\times M$ such that $\pazocal{U}$ is stable by the backward flow, i.e.~$e^{-t\euler}\pazocal{U}\subset \pazocal{U}$  for all $t\in \mathbb{R}_{\geqslant 0}$.
\end{lemm}

The diagonal $\Delta\subset M\times M$ 
is a critical manifold of $\euler$ 
and is preserved by the flow, and $\pazocal{U}$ 
is the \emph{unstable manifold} 
of $\Delta$ in the terminology 
of dynamical systems. 
The vector field $\euler$ is 
\emph{hyperbolic} in the normal direction $N\Delta$ 
as we will next see.

\begin{refproof}{Lemma \ref{l:stable}}  The idea is to observe that by definition of an Euler vector field $V$, near any $p\in \Delta$ we can choose an arbitrary coordinate frame $(x^i,h^i)$ such that $\Delta$ is locally given by the equations $\{h^i=0\}$ and $\euler=(h^i+A_i(x,h))\partial_{h^i}$ where $A_i\in \pazocal{I}^2$. The fact that there is no component in the direction $\partial_{x^i}$
comes from the fact that our vector field $\euler$ preserves the fibration with leaves $x=\text{constant}$. 

Fix a compact $K\subset M$ and consider the product $K\times M $,
which contains $\Delta_K=\{(x,x)\in M^2 \st x\in K\}$ and is preserved by the flow.
For the moment we work in $K\times M$ and we conclude a global statement later on. 
We also choose some Riemannian metric $g$ on $M$ and consider the  smooth function germ $M^2  \ni (m_1,m_2)\mapsto \mathbf{d}_g^2(m_1,m_2)\in \mathbb{R}_{\geqslant 0} $
defined near the diagonal $\Delta_K\subset K\times M$, where $\mathbf{d}_g$ is the distance function. 
In the local coordinate frame $(x^i,h^i)_{i=1}^n$ defined near $p$, $\mathbf{d}^2$ reads
$$ \mathbf{d}^2((x,0),(x,h)) = A_{ij}(x)h^ih^j+\pazocal{O}(\vert h\vert^3) $$ 
where $A_{ij}(x)$ is a positive definite matrix. Thus setting $f=\mathbf{d}^2$ yields
$\euler f=2f+\pazocal{O}(\vert h\vert^3)$ by definition of $\euler$
and
therefore there exists some
$\varepsilon>0$ such that $\forall (x,h)\in K\times M, f\leqslant \varepsilon\implies \euler f\geqslant 0$.
Observe that $\euler \log f=2+\pazocal{O}(\mathbf{d}_g) $, $\euler \log(f)|_{\Delta_K}=2$ and $\euler \log(f)$ is continuous near $\Delta_K$. 
By compactness of $K$, there exists some $\varepsilon>0$ s.t. if $f\leqslant \varepsilon$
then $\euler \log(f)\geqslant \frac{3}{2}$.
We take  $\pazocal{U}_K=\{f\leqslant \varepsilon\}\cap K\times M$.

The vector field $\euler$ vanishes on $\Delta$ therefore the flow $e^{-t\euler}$
preserves $\Delta$.
Assume there exists $(x,h)\in \pazocal{U}_K\setminus \Delta_K$ such that $e^{-T\euler}(x,h)\notin \pazocal{U}_K$ for some $T> 0$.
Without loss of generality, we may even assume that $f(x,h)=\varepsilon$.
Then, 
let us denote  $T_1=\inf\{ t \st  t>0,\ f(e^{-t\euler}(x,h))=\varepsilon  \}$ which is intuitively 
the first time for which $f(e^{-T_1\euler}(x,h))=f(x,h)=\varepsilon$. 
Since $(x,h)\notin \Delta_K$,  we have $-\euler \mathbf{d}^2(x,h)\leqslant -\frac{3}{2} \mathbf{d}^2(x,h)<0$ and setting $f=\mathbf{d}^2$
yields
$$ f(e^{-t\euler}(x,h))=f(x,h)-t\euler f(x,h)+\pazocal{O} (t^2) $$ which means that $f(e^{-t\euler}(x,h))$ is strictly decreasing near $t=0$, hence  necessarily  $T_1>0$. By the fundamental theorem of calculus,
$$f(e^{-T_1\euler}(x,h))-f(x,h)=\int_0^{T_1} -{\left(\euler f\right)}(e^{-s\euler}(x,h))ds  $$
and since 
$$ -{\left(\euler f\right)}(e^{-s\euler}(x,h))\leqslant -\frac{3}{2}f(e^{-s\euler}(x,h) ) <0$$ for all $s\in [0,T_1]$, we conclude that
$f(e^{-T_1\euler}(x,h))<f(x,h)$ which yields a contradiction. So 
for all compact $K\subset M$, we found a neighborhood $\pazocal{U}_K\subset K\times M$ of $\Delta_K$ (for the induced topology) which is stable by $e^{-t\euler},t\geqslant 0$. Then by paracompactness of $M$, we can take a locally finite subcover of $\Delta$ by such sets and we deduce the existence of a global neighborhood $\pazocal{U}$ of $\Delta$ which is stable by
$e^{-t\euler},t\geqslant 0$.
\end{refproof}

{
In the present section, instead of using charts, we  
favor a presentation using coordinate frames,
which makes notation
simpler. The two viewpoints are equivalent
since given a chart $\kappa:U\subset \pazocal{M}\to \kappa(U)\subset \mathbb{R}^n$ on some smooth manifold $\pazocal{M}$ of dimension $n$, the linear 
coordinates $(x^i)_{i=1}^n \in \mathbb{R}^{n*}$ on $\mathbb{R}^n$ can be pulled back on 
$U$ as a coordinate frame $(\kappa^*x^i)_{i=1}^n\in C^\infty(U;\mathbb{R}^n)$.  
}

The next proposition  gives a normal form for Euler vector fields. 

\begin{prop}[Normal form for Euler vector fields]\label{p:normalform}
Let $\euler$ be an Euler vector field. 
There exists a unique subbundle
$N\Delta\subset T_\Delta \left(M\times M\right)$, such that
$ de^{t\euler}=e^{t}\id: N\Delta\to N\Delta$. 

For all $p\in \Delta$, there exist  coordinate functions $(x^i,h^i)_{i=1}^n$ defined near $p$ such that { in these local coordinates} near $p$, $\Delta=\{ h^i=0 \}$ and
$
\euler=\sum_{i=1}^n h^i\partial_{h^i}$ $\forall i\in \{1,\dots,n\}$.
\end{prop}

\begin{rema}
This result was proved in \cite{dangthesis} and also later 
in the paper by Bursztyn--Lima--Meinrenken \cite{Bursztyn2019}, cf.~the review \cite{Meinrenken2021}. Our proof here is different and more in the spirit of the Sternberg--Chen linearization theorem.
\end{rema}

\begin{proof} 
\step{1}
We prove the dynamics contracts exponentially fast. We  use the distance function $f=\mathbf{d}^2$ and note that
$-\euler  \log(f)\leqslant -\frac{3}{2} $ on the open set $\pazocal{U}$ constructed in Lemma~\ref{l:stable} therefore
$ { e^{-t\euler}}f\leqslant e^{-\frac{3}{2}t}f $ by Gronwall Lemma.
Consequently, there exists a neighborhood $\pazocal{U}$ of $\Delta$ s.t. for any function $f\in \pazocal{I}$ ($f$ vanishes on 
the diagonal $\Delta$) and $U$ is some bounded open subset, we have the exponential decay
$\Vert { e^{-t\euler}}f\Vert_{L^\infty(U)} \leqslant C e^{-Kt}
 $ for some $C>0,  K>\frac{1}{2}$ due to the hyperbolicity
in the normal direction of $e^{-t\euler}$.
Moreover, Hadamard's lemma states that if $f\in \pazocal{I}^k$ which means $f$ vanishes of order $k$, then 
locally we can always write
$f$ as $\sum_{\vert \beta\vert = k} h^\beta { g_\beta}(x,h) $
where $h\in \pazocal{I}$ and therefore gluing with a partition of unity
yields a decay estimate of the form 
$$\Vert { e^{-t\euler}}f\Vert_{L^\infty(U)} \leqslant C e^{-Kkt}$$
where $C>0$ and we have better exponential decay. 
So starting from the coordinates $(x^i,h^i)$ from {the proof of Lemma~\ref{l:stable}}, we will correct the coordinates
$(h^i)_{i=1}^n$ using the exponential contractivity of the flow
to obtain normal forms coordinates.

\step{2 }We now correct $h^i$ so that 
$\euler h^i=h^i$ modulo an element in $\pazocal{I}^\infty$.
First observe that $\euler h^i-h^i\in \pazocal{I}^2$ by definition, therefore 
setting $ h_1^i= h^i +\varepsilon^i_1$, $\varepsilon^i_1= -\frac{(\euler h^i-h^i)}{2}  $, we verify that
\begin{eqnarray}
 \euler h_1^i-h_1^i   \in \pazocal{I}^3. 
 \end{eqnarray}
By recursion, we define a sequence $(h_k^i)_{i=1}^n, k\in \mathbb{N}$,  
defined as
$h_{k+1}^i=h_k^i+\varepsilon_{k+1}^i$ 
where $\varepsilon_{k+1}^i= -\frac{(\euler h^i_k-h^i_k)}{k+2}  $ 
and we verify that for all $k\in \mathbb{N}$, we have $ \euler h_k^i-h_k^i   \in \pazocal{I}^{k+2}$.
By Borel's Lemma, we may find a smooth germ 
$h_\infty^i\sim h^i + \sum_{k=1}^\infty \varepsilon_k^i $ hence we
deduce that there exists $(h_\infty^i)_{i=1}^n$ s.t. 
$ \euler h_\infty^i-h_\infty^i   \in \pazocal{I}^\infty $. 

\step{3} We use the flow to make the coordinate functions $(h_\infty^i)_{i=1}^n$ exact solutions of
$\euler f=f$.
Set $$\tilde{h}^i=h^i_\infty-\int_0^\infty e^{t} { \left( e^{-t\euler}\left((\euler-1)h^i_\infty \right)\right) }dt$$
where the integrand converges absolutely since $(\euler-1)h^i_\infty\in \pazocal{I}^\infty $, hence ${ e^{-t\euler}\left((\euler-1)h^i_\infty \right)}=\pazocal{O}(e^{-tNK}) $ for all $N>0$ where $K>\frac{1}{2}$. The function $\tilde{h}^i$ is smooth since
the ideal $\pazocal{I}^\infty$ is stable by derivatives therefore
differentiating under the integral $\int_0^\infty e^{t} { \left( e^{-t\euler}\left((\euler-1)h^i_\infty \right)\right) }dt$ 
does not affect the decay of the integral. 
So we obtain that for all $i\in \{1,\dots,n\}$, $\euler \tilde{h}^i=\tilde{h}^i$ which 
solves the problem since $(x^i,\tilde{h}^i)$ is a germ of smooth coordinate frame near $p$.
\end{proof}

\subsection{Log-polyhomogeneity}

Let $\euler$ be an Euler vector field.
One says that a distribution $u\in \pazocal{D}^\prime(\pazocal{U})$
is \emph{weakly 
homogeneous of degree $s$} w.r.t.~scaling with $\euler$ if the family $(e^{ts}  { (e^{-t\euler}u)})_{t\in \mathbb{R}_{\geqslant 0}}$ 
is bounded in $\pazocal{D}^\prime(\pazocal{U})$ (cf.~Meyer \cite{Meyer}). One can also introduce a more precise variant of that definition by replacing  $\pazocal{D}^\prime(\pazocal{U})$ with $\pazocal{D}^\prime_\Gamma(\pazocal{U})$ for some closed conic $\Gamma\subset T^*M^2\setminus\zero$, where $\pazocal{D}^\prime_\Gamma(\pazocal{U})$  is Hörmander's space of distributions with wavefront set in $\Gamma$ (see \cite[\S8.2]{H} for the precise definition).  As shown in \cite[Thm.~1.4]{DangAHP}, {in the first situation without the wavefront condition, this defines} a class of distributions that is intrinsic, i.e.~which does not depend  on the choice of  Euler vector field $\euler$.

We consider  distributions with the following log-polyhomogenous behaviour under scaling transversally to the diagonal.

\begin{defi}[log-polyhomogeneous distributions]
Let $\Gamma$ be a closed conic set 
such that for some $\euler$-stable neighborhood $\pazocal{U}$ of the diagonal, 
\begin{eqnarray}
\forall t\geqslant 0, \ { e^{-t\euler}}\Gamma|_\pazocal{U}\subset \Gamma|_{\pazocal{U}},\\
\overline{\Gamma}\cap T^*_\Delta M^2=N^*\Delta.
\end{eqnarray}
We say that $u\in \pazocal{D}^\prime_\Gamma(\pazocal{U})$ is \emph{log-polyhomogeneous} w.r.t.~$\euler$ if
it admits the following asymptotic expansion under scaling: there exists $p\in \mathbb{Z}$, {$l\in \nn$} 
and distributions $(u_k)_{k=p}^{\infty}, 1\leqslant i \leqslant l $  
in $\pazocal{D}^\prime_\Gamma(\pazocal{U})$ such that for all $N>0$ and all $\varepsilon>0$,
\begin{eqnarray}\label{phexp}
{ e^{-t\euler}}u=\sum_{p\leqslant  k\leqslant N, 0\leqslant i\leqslant l-1}  e^{-tk} \frac{(-1)^i t^i}{i!}\left(\euler-k\right)^i u_k +\pazocal{O}_{\pazocal{D}^\prime_\Gamma(\pazocal{U})}(e^{-t(N+1-\varepsilon)}).
\end{eqnarray}
A distribution is called \emph{polyhomogeneous} if $l=0$. { In contrast, a  non-zero value for $l$ indicates the occurrence of \emph{logarithmic mixing} under scaling}.

We endow such {distributions} with a notion of convergence as follows: a sequence of
log-polyhomogeneous {distributions} $u_n$ converges  
$u_n\to v$ in log-polyhomogeneous {distributions} if 
$u_n\rightarrow v$ in $\pazocal{D}^\prime_\Gamma(M)$, for every $N$ each term in the asymptotic expansion converge 
$u_{n,k}\rightarrow v_{k}$, $k\leqslant N $ and the remainders $u_n-\sum_{k=p}^N u_{n,k} $ converge to $v-\sum_{k=p}^N v_{k} $ in the sense that
$$
 { e^{-t\euler}}\bigg(u_n-\sum_{k=p}^N u_{n,k}- \Big(v-\sum_{k=p}^N v_{k} \Big)\bigg)=\pazocal{O}_{\pazocal{D}^\prime_\Gamma(\pazocal{U})}(e^{-t(N+1-\varepsilon)})
 $$
  for all 
$\varepsilon>0$. 
\end{defi}

Thus, log-polyhomogeneous distributions have resonance type expansions under scaling with the vector field $\euler$. 
We stress, however, that each distribution $u_k$  in the expansion \eqref{phexp} is not necessarily homogeneous. 
In fact, it does
not necessarily scale like ${ e^{-t\euler}}u_k=e^{-tk}u_k$, but we may have
logarithmic mixing in the sense that:
\begin{eqnarray*}
{ e^{-t\euler}}u_k=\sum_{i=0}^{l-1} e^{-tk} \frac{(-1)^i t^i}{i!}\left(\euler-k\right)^i u_k. 
\end{eqnarray*}
This means that restricted to the linear span of $(u_k,(\euler-k) u_k,\dots,(\euler-k)^{l-1}u_k)$, the matrix of $\euler$ 
reads
\begin{eqnarray*}
\euler\left(\begin{array}{c}
u_k\\
(\euler-k) u_k\\
\vdots \\
(\euler-k)^{l-1}u_k
\end{array} \right)=\left(\begin{array}{cccc}
k&1&  &0\\
&k& \ddots &\\
&   &\ddots &1  \\
0 &  &   &k
\end{array} \right)\left(\begin{array}{c}
u_k\\
(\euler-k) u_k\\
\vdots \\
(\euler-k)^{l-1}u_k
\end{array} \right)
\end{eqnarray*}
so it has a Jordan block structure.

In the present paper, we will prove that  log-polyhomogeneous distributions which are
Schwartz kernels of pseudodifferential operators with classical symbols as well as Feynman propagators   have no Jordan blocks for the resonance $p\leqslant  k<0$ and there are Jordan blocks of rank $2$ for all $k\geqslant 0$. In other words, 
$(u_k,(\euler-k) u_k, (\euler-k)^2u_k) $ are linearly dependent of rank $2$ for every $k\geqslant 0$.
We introduce special terminology to emphasize this type of behaviour.

\begin{defi}[Tame log-polyhomogeneity]
A distribution $u\in \pazocal{D}^\prime_\Gamma(\pazocal{U})$ is \emph{tame log-polyhomogeneous} w.r.t.~$\euler$ 
if it is log-polyhomogeneous  w.r.t.~$\euler$  and 
\begin{eqnarray}
{ e^{-t\euler}}u=\sum_{p\leqslant  k<0}  e^{-tk}  u_k + \sum_{0\leqslant  k\leqslant N, 0\leqslant i\leqslant 1}  e^{-tk} \frac{(-1)^i t^i}{i!}\left(\euler-k\right)^i u_k +\pazocal{O}_{\pazocal{D}^\prime_\Gamma(\pazocal{U})}(e^{-t(N+1-\varepsilon)})
\end{eqnarray}
for all $\varepsilon>0$, i.e.~the Jordan blocks only occur for non-negative $k$ and have rank at most $2$.
\end{defi}

For both pseudodifferential operators with classical symbols and Feynman powers, 
we will prove that the property of being log-polyhomogeneous is \emph{intrinsic} and does not depend on the 
choice of Euler vector field used to define the log-polyhomogeneity. This generalizes the fact that the class of pseudodifferential operators 
with polyhomogeneous symbol is intrinsic.

\subsection{Pollicott--Ruelle resonances of \texorpdfstring{$e^{-t\euler}$}{exp(-tX)} acting on log-polyhomogeneous distributions}

For every tame log-polyhomogeneous distribution $\exd \in  \cD^\prime(\pazocal{U})$ and every $n\in \mathbb{Z}$, we  define a projector
$\Pi_n$ which extracts the quasihomogeneous part $\Pi_n(\exd) \in  \cD^\prime(\pazocal{U})$ of the distribution $\exd$.

Note that if a distribution $u$ is log-polyhomogenous w.r.t.~$\euler$, then for any test form $\varphi\in \Omega^{\bullet}_{\rm c}(\pazocal{U})$~\footnote{{We consider test forms because  Schwartz kernels of operators are not densities and it is appropriate to consider them as differential forms of degree $0$.}} 
where $\pazocal{U}$ is $\euler$-stable, 
we have an asymptotic expansion:
\begin{eqnarray*}
{\left\langle { (e^{-t\euler}u)},\varphi\right\rangle} = \sum_{k=p,0\leqslant i\leqslant l-1}^N e^{-tk}\frac{(-1)^it^i}{i!} \left\langle (\euler-k)^iu_{k},\varphi\right\rangle+\pazocal{O}(e^{-tN}). 
\end{eqnarray*}
The l.h.s.~is similar to dynamical correlators studied in dynamics and the asymptotic expansion is similar 
to expansions of dynamical correlators in hyperbolic dynamics.
So in analogy with dynamical system theory, we can define the Laplace transform
of the dynamical correlators and the Laplace transformed 
correlators have meromorphic continuation to the complex plane with poles
along the arithmetic 
progression $\{p,p+1,\dots \}$:
\[
\int_0^\infty e^{-tz}\left\langle { (e^{-t\euler}u)},\varphi\right\rangle dt=
\sum_{k=p,0\leqslant i\leqslant l-1}^N (-1)^i \frac{\left\langle (\euler-k)^i u_{k},\varphi\right\rangle}{(z+k)^{{i+1}}}+\text{holomorphic on }\Re z\leqslant N.
\]
These poles are \emph{Pollicott--Ruelle resonances} of the flow $e^{-t\euler}$ acting on 
log-polyhomogeneous distributions in $\pazocal{D}^\prime(\pazocal{U})$. 

 We can now use the Laplace transform to define
the projector $\Pi_n$ which extracts quasihomogeneous parts of distributions.
\begin{defi} \label{def:pi}
Suppose $u\in \cD^\prime(\pazocal{U})$ is log-polyhomogeneous. Then for $n\in\zz$ we define 
$$
\Pi_n(u)\defeq \frac{1}{2i\pi}\int_{\partial D} \mathfrak{L}_zu \,dz 
$$ 
where $ \mathfrak{L}_z u= \int_0^\infty e^{-tz}{ (e^{-t\euler}u)} \,dt$ and $D\subset \mathbb{C}$ is a small disc around $n$.
\end{defi}

\subsection{Residues as homological obstructions and scaling anomalies}\label{ss:toy}

%

Before considering the general setting, let us explain the concept of residue in the following fundamental example  (which is closely related to the discussion in the work of Connes--Moscovici~\cite[\S5]{Connes1995}, Lesch \cite{Lesch}, Lesch--Pflaum \cite{Lesch2000}, Paycha \cite{paycha2,paycha} and Maeda--Manchon--Paycha \cite{Maeda2005}).

Let $\Feuler\in {\cf(T\rr^{n})}$ be an \emph{Euler vector field with respect to $0\in \rr^{n}$}, i.e.~for all $f\in\cf(\rr^{n})$, $\Feuler f-f$ vanishes at $0$ with order $2$. For instance, we can consider $\Feuler=\sum_{i=1}^n \xi^i \partial_{\xi^i}$, where $(\xi^1,\dots,\xi^n)$ are the Euclidean coordinates. This simplified setting is meant to illustrate what happens on the level of \emph{symbols} or \emph{amplitudes} rather than Schwartz kernels near $\Delta\subset M\times M$, but these two points of view are very closely related. In our toy example, this simply corresponds to the relationship between momentum variables $\xi^{i}$ and position space variables $h^{i}$ by inverse Fourier transform, see Remark \ref{l:scalanomal}.

 Suppose $\exd\in \pazocal{D}^{\prime,n}(\rr^n\setminus\{0\})$ is a de Rham current of top degree
which solves the linear PDE:
\beq\label{toextend}
\Feuler \exd=0 \mbox{ in the sense of } \pazocal{D}^{\prime,n}(\rr^n\setminus \{0\}),
\eeq
which
means that the current $\exd$ is scale invariant on $\rr^n\setminus \{0\}$.
\begin{lemm}
Under the above assumptions, $\iota_\Feuler \exd$ is a closed current in 
{$\pazocal{D}^{\prime,n-1}(\rr^n\setminus \{0\})$} where $\iota_\Feuler$ denotes the contraction with $\Feuler$.
\end{lemm}
\begin{proof}
The current $\iota_\Feuler \exd$ is closed in $\pazocal{D}^{\prime,n-1}(\rr^n\setminus \{0\}) $
 by the 
Lie--Cartan formula $\left(d\iota_\Feuler+\iota_\Feuler d \right)=\Feuler$ and the fact that $\exd$ is closed as a top degree current:
\[ 
d\iota_\Feuler {\exd}=\left(d\iota_\Feuler+\iota_\Feuler d \right)\exd=\Feuler \exd=0.\vspace{-1em}
\]
\end{proof}

One can ask the question: \emph{is there a distributional extension
$\overline{\exd}\in  \pazocal{D}^{\prime,n}(\rr^n)$ of $\exd$ which satisfies the same scale invariance PDE on $\rr^n$? } The answer 
is positive unless there is an obstruction of cohomological nature which we explain in the following proposition.

\begin{prop}[Residue as homological obstruction]\label{p:residue}
Suppose $\exd\in \pazocal{D}^{\prime,n}(\rr^n\setminus\{0\})$ satisfies \eqref{toextend}. Let $\chi\in C^\infty_\c(\rr^n)$ be such that $\chi=1$ near $0$.
Then $d\chi$ is an exact form and
the pairing between the exact form $d\chi$ and the closed current $\iota_\Feuler \exd $
$$
\left\langle d\chi, \iota_\Feuler \exd \right\rangle=\int_{\rr^n} d\chi\wedge \iota_\Feuler \exd
$$
{does not depend} on the choice of $\chi$.

If moreover $\wf(\exd)\subset \{ (\xi,\tau dQ(\xi))  \st   Q(\xi)=0,\ \tau<0\}$ for some non-degenerate quadratic form $Q$ on $\rr^n$, then
$$
\int_{\mathbb{S}^{n-1}} \iota_\Feuler \exd=\left\langle d\chi, \iota_\Feuler \exd \right\rangle.
$$

There is a scale invariant extension $\overline{\exd}$ of $u$ if and only if 
the pairing $\left\langle d\chi, \iota_\Feuler \exd \right\rangle=0$, which is equivalent to saying that
the current $\iota_\Feuler \overline{\exd}\in \pazocal{D}^{\prime,n-1}(\rr^n)$ is closed. 
\end{prop} 

\begin{proof}
Since $\iota_\Feuler \exd$ is closed and $d\chi$ is exact the cohomological pairing 
$\left\langle d\chi, \iota_\Feuler \exd \right\rangle
$ does not depend on the choice of $\chi$. In fact, as a de Rham current $d\chi\in \pazocal{D}^{\prime,1}(\rr^n)$ lies in the same cohomology class as the current $[\mathbb{S}^{n-1}]\in \pazocal{D}^{\prime,1}(\rr^n)$ of integration on a sphere $\mathbb{S}^{n-1}$ enclosing $0$.   

If there is an extension $\overline{\exd}$ that satisfies $\Feuler \overline{\exd}=0$ in  
$\pazocal{D}^{\prime,n}(\rr^n)$, it means that the current $\iota_\Feuler \overline{\exd}$ is closed in $\pazocal{D}^{\prime,n-1}(\rr^n) $
since $ d\iota_\Feuler \overline{\exd}=\left(d\iota_\Feuler+\iota_\Feuler d \right)\overline{\exd}=\Feuler \overline{\exd}=0 $.
Then by integration by parts (sometimes called the Stokes theorem for de Rham currents),  
$\left\langle d\chi, \iota_\Feuler \exd \right\rangle=\left\langle d\chi, \iota_\Feuler \overline{\exd} \right\rangle=-\left\langle \chi, \Feuler \overline{\exd} \right\rangle= 0$
where we used the fact that $d\chi$ vanishes near $0$ and $\exd=\overline{\exd}$ in a neighborhood of the support of $d\chi$.

Conversely, assume the cohomological pairing vanishes: $\left\langle d\chi, \iota_\Feuler \exd \right\rangle=0
$. Let $\overline{\exd}$ be any extension of $\exd$. Then $\left\langle \chi, \Feuler \overline{\exd} \right\rangle=0
$ by integration by parts. But since $\exd=\overline{\exd}$ outside $0$ and $\Feuler \exd=0$ outside $0$, 
the current $\Feuler \overline{\exd}$ is supported at $0$ and by a classical theorem of Schwartz must have the form
$$\Feuler \overline{\exd}=\bigg( c_0\delta_{\{0\}}(\xi) + \sum_{1\leqslant\vert \alpha \vert\leqslant N} c_\alpha\partial_\xi^\alpha\delta_{\{0\}}(\xi) \bigg) d\xi^1\wedge \dots\wedge d\xi^n   $$ 
where all $\alpha$ are multi-indices and $N$ is the distributional 
order of the current. 
Since $\chi=1$ near $0$, it means 
$\left\langle \chi, \Feuler \overline{\exd} \right\rangle=0
=c_0\chi(0)=c_0=0$ hence the constant term vanishes.
This means that $\Feuler \overline{\exd}= \sum_{1\leqslant\vert \alpha \vert\leqslant N}c_\alpha\partial_\xi^\alpha\delta_{\{0\}}(\xi)  d\xi^1\wedge \dots\wedge d\xi^n$ and $\overline{\exd}- \sum_{1\leqslant\vert \alpha \vert\leqslant N}\frac{c_\alpha}{\vert\alpha\vert}\partial_\xi^\alpha\delta_{\{0\}}(\xi)  d\xi^1\wedge \dots\wedge d\xi^n$ extends $\exd$ and $\Feuler\left(\overline{\exd}- \sum_{1\leqslant\vert \alpha \vert\leqslant N}\frac{c_\alpha}{\vert\alpha\vert}\partial_\xi^\alpha\delta_{\{0\}}(\xi)  d\xi^1\wedge \dots\wedge d\xi^n \right)=0$.

When $\wf(\exd)\subset \{ (\xi,\tau dQ(\xi)) \st  Q(\xi)=0, \ \tau<0\}$ then $\wf(\exd)$ does not meet the conormal of $\mathbb{S}^{n-1}$
and therefore we can repeat the exact above discussion with the indicator function $\one_B$ of the unit ball $B$ playing the role of $\chi$,
since the distributional product $\one_B \exd$ is well-defined because $\wf(\one_B)+ \wf(\exd)$ never meets the zero section.
Then we obtain the residue from the identity $\partial \one_B=[\mathbb{S}^{n-1}]$ for currents
where $[\mathbb{S}^{n-1}]$ is the integration current on the sphere $\mathbb{S}^{n-1}$. 
\end{proof}

The quantity $\left\langle d\chi, \iota_\Feuler \exd \right\rangle=\left\langle [\mathbb{S}^{p-1}],[\iota_\Feuler \exd] \right\rangle$,
called \emph{residue} or \emph{residue pairing}, measures a cohomological obstruction to extend $\exd$ as a solution $\overline{\exd}$ solving $\Feuler \overline{\exd}=0$.
In fact, a slight modification of the previous proof  shows that there is always an
extension $\overline{\exd}$ which satisfies the linear PDE
$$
\boxed{\Feuler \overline{\exd}=\left\langle d\chi, \iota_\Feuler \exd \right\rangle \delta_{\{0\}}  d\xi^1\wedge \dots\wedge d\xi^n.} 
$$

We show  a useful vanishing {property} of  {certain}  residues.

\begin{coro}[{Residue vanishing}]\label{l:vanishing1}
{ Let $Q$ be a nondegenerate quadratic form on $\mathbb{R}^n$.}
Suppose $\exd\in \pazocal{D}^\prime(\mathbb{R}^n\setminus \{0\})$ is homogeneous of degree $-n+k>-n$ and $$\wf(\exd)\subset \{ (\xi,\tau dQ(\xi)) \st  Q(\xi)=0, \ \tau<0\}.$$ 
Then for every
multi-index $\beta$ such that $\vert\beta\vert=k>0$,   
$$
\int_{\mathbb{S}^n} \big(\partial_{\xi}^\beta \exd\big) \iota_\Feuler d\xi_1\dots d\xi_n =0. 
$$
\end{coro}
\begin{proof}
Let $\one_B$ be the indicator function of the unit ball $B$. We denote by $\overline{\exd}$, the unique distributional extension of $\exd\in \pazocal{D}^\prime(\mathbb{R}^n\setminus \{0\})$ in $\pazocal{S}^\prime(\mathbb{R}^n)$ which is homogeneous of degree $-n+k$  by~\cite[Thm.~3.2.3, p.~75]{H}. 
Therefore using the commutation relation $[\Feuler,\partial_\xi^\beta]=-\vert\beta\vert=-k$ yields immediately
that $\partial_\xi^\beta \overline{\exd}$ is a distribution
homogeneous of degree $-n$ and thus $\Feuler \left(\partial_\xi^\beta \overline{\exd}\,d^n\xi\right)=0$.
Then, by Proposition~\ref{p:residue}, the residue equals
$$ \int_{\mathbb{S}^{n-1}} \big(\partial_{\xi}^\beta \exd\big) \iota_\Feuler d\xi_1\dots d\xi_n=\int_{\mathbb{R}^n} \left(\partial \one_B\right) \iota_\Feuler \partial_\xi^\beta \overline{\exd}\,d^n\xi=0, $$
where the pairing is well-defined since $N^*(\mathbb{S}^{n-1})\cap \wf(\exd)=\emptyset$.
\end{proof}

\begin{remark}[Residue as scaling anomaly]\label{l:scalanomal}
Let $\exd\in \pazocal{D}^{\prime,n}(\rr^n\setminus\{0\})$ be a current of top degree, homogeneous of degree $0$ with respect to scaling and denote by $\overline{\exd}\in \pazocal{D}^{\prime,n}(\rr^n)$ its unique distributional extension of order $0$.
Denote by $\big(\pazocal{F}^{-1}\exd\big)(h)=\frac{1}{(2\pi)^n}\left\langle \overline{\exd}, e^{i\left\langle h, . \right\rangle}\right\rangle\in \pazocal{S}^\prime(\rr^n)$ its inverse Fourier transform.

Then the tempered distribution $\pazocal{F}^{-1}\exd$ satisfies the equations:
\begin{eqnarray*}
\pazocal{F}^{-1}\exd(\lambda.)=\pazocal{F}^{-1}\exd(.)+c\log\lambda,\\
\euler \pazocal{F}^{-1}\exd=c,
\end{eqnarray*} 
where $\euler=\sum_{i=1}^n h^i\partial_{h^i}$ is the Euler vector field in position space and $c=\frac{1}{(2\pi)^n}\int_{\rr^n} d\chi\wedge \iota_\Feuler \exd$ is the residue. Therefore, residues defined as homological obstructions also arise as \emph{scaling anomalies}.
\end{remark} 

   This interpretation of residues as scaling anomalies have appeared in the first author's  thesis \cite[\S8]{dangthesis} as well as in the  physics literature on renormalization in Quantum Field Theory in the Epstein--Glaser approach  \cite{Nikolov2013,Gracia-Bondia2014,Rejzner2020}.

\subsection{Dynamical definition of residue} After this motivation, we come back to the setting of an Euler vector field $\euler$ acting on a neighborhood of the diagonal $\Delta\subset M\times M$.

As we will explain, our approach to the Wodzicki residue uses scalings with Euler vector fields and a diagonal restriction.
Let $\iota_\Delta: x\mapsto (x,x)\in \Delta\subset M\times M$
denote the diagonal embedding.
We are  ready to formulate  our main definition. 
\begin{defi}[Dynamical residue]
Let  ${\pazocal{K}}\in \pazocal{D}^\prime_\Gamma(\pazocal{U})$ be a tame log-polyhomogeneous  distribution on some neighborhood $\pazocal{U}$ of the diagonal $\Delta\subset M\times M$
and suppose  $\Gamma|_{\Delta}\subset N^*\Delta$.
For any Euler vector field $\euler$, let $\Pi_0$
be the corresponding spectral projector on the resonance $0$, see Definition \ref{def:pi}.  We define  
the \emph{dynamical residue of ${\pazocal{K}}$}  as: 
$$
\resdyn {\pazocal{K}} = \iota_\Delta^*\big(\euler (  \Pi_0( {\pazocal{K}}) )\big) \in C^\infty(M),
$$
provided that the pull-back is well-defined.
\end{defi}

A priori, the dynamical residue can depend on the choice of Euler vector field $\euler$ and  it is not obvious that one can pull-back the distribution $\euler (\Pi_0({\pazocal{K}}))$ by the diagonal embedding. We need therefore to examine the definition carefully for classes of Schwartz kernels that are relevant for complex powers of differential operators.


\section{Equivalence of definitions  in pseudodifferential case}\label{ss:wodzickipdo}

\subsection{Log-polyhomogeneity of pseudodifferential operators}


In this section, $M$ is a smooth manifold of arbitrary dimension. 

{We denote by $\vert \Lambda^{\top}M \vert$ the space of smooth densities
on $M$. For any operator $A:C^\infty_\c(M)\to \cD^\prime(M)$, recall that the corresponding Schwartz kernel is a distribution on $M\times M$ twisted by some smooth density. More precisely, the kernel  of $A$ belongs to $\cD^\prime(M\times M)\otimes \pi_2^*\vert \Lambda^{\top}M \vert$ where $\pi_2$ is the projection
on the second factor and reads ${\pazocal{K}}(x,y)\dvol_g(y)$ where ${\pazocal{K}}\in \cD^\prime(M\times M)$ and $\dvol_g\in \vert \Lambda^{\top}M \vert$~\footnote{ In fact, $Au=\int_{y\in M} {\pazocal{K}}(.,y)u(y)\dvol_g(y)$ $\forall u\in C^\infty_\c(M)$. Neither ${\pazocal{K}}\in \cD^\prime(M\times M)$ nor $\dvol_g\in \vert \Lambda^{\top}M \vert$ are intrinsic, but their product is.}.}

{ In this part, we need to fix a density $\dvol_g\in \vert \Lambda^{\top}M \vert$ on our manifold $M$ since given a linear continuous operator from $C^\infty_\c(M)\to \pazocal{D}^\prime(M)$, its Schwartz kernel $ \pazocal{K}$ and hence its dynamical residue depends on the choice of density. However, we will see in the sequel that the product: (dynamical residue $\times$ density) $\in \vert \Lambda^{\top}M \vert$ does not depend on the choice of density.}

We first prove that pseudodifferential kernels
are tame log-polyhomogeneous with respect to \emph{any} Euler vector field $\euler$.
\begin{prop}\label{p:pseudopoly}
Let  ${\pazocal{K}}(.,.)\pi_2^*\dvol_g\in \pazocal{D}^\prime_{N^*\Delta}(M\times M)\otimes \pi_2^*\vert \Lambda^{\top}M \vert$ be the kernel of a classical pseudodifferential operator $A\in \Psi^\cv_{\ph}(M)$, {$\cv\in \mathbb{C}$}. 
Then for every Euler vector field $\euler$, there exists an $\euler$-stable neighborhood of the diagonal $\pazocal{U}$ such that ${\pazocal{K}}$ is tame log-polyhomogeneous w.r.t.~$\euler$.

In particular,
{
$$
\mathfrak{L}_s{\pazocal{K}} { := \int_0^\infty \left(e^{-t(\euler+s)}{\pazocal{K}}\right) dt \in \pazocal{D}^\prime_{N^*\Delta}(\pazocal{U})}
$$}
is a well-defined conormal distribution and extends as a {meromorphic function} of $s\in \mathbb{C}$
with poles at $s\in \cv+n-\mathbb{N}$.

{ If $\cv\geqslant -n$ is an integer,
the poles at $s=k$ are simple when $k<0$ and of multiplicity at most $2$ when
$k\geqslant 0$. If $\cv\in \mathbb{C}\setminus \clopen{-n,+\infty}\cap \mathbb{Z} $ then all poles are simple and $\Pi_0(\pazocal{K})=0$.}
\end{prop}
{In the proof, we make a crucial use of the Kuranishi trick, which allows us to represent a pseudodifferential kernel in normal form coordinates for a given Euler vector field $X$.
Concretely, in local coordinates, the phase term used to represent the pseudodifferential kernel as an oscillatory integral reads $e^{i\left\langle\xi,x-y \right\rangle}$, yet we would like to write it in the form $e^{i\left\langle\xi,h\right\rangle}$ where $X=\sum_{i=1}^nh^i\partial_{h^i}$. We also need to study how the symbol transforms in these normal form coordinates and to verify that it is still polyhomogeneous in the momentum variable $\xi$. 
Our proof can be essentially seen as a revisited version of the theorem of change of variables for pseudodifferential operators combined with scaling of  polyhomogeneous symbols.}

\begin{refproof}{Proposition \ref{p:pseudopoly}}
\step{1} Outside the diagonal the 
Schwartz kernel ${\pazocal{K}}$ is smooth, hence for any test form $\chi_1\in C^\infty_{\rm c}(M\times M\setminus \Delta)$ and any smooth function
$\psi\in C^\infty(M\times M)$ supported away from the diagonal,
$$\left\langle { e^{-t\euler}} ( {\pazocal{K}} \psi) ,\chi\right\rangle=\pazocal{O}((e^{-t})^{+\infty}).$$ 
This shows we only need to prove the tame log-polyhomogeneity for a localized version of the kernel near the diagonal $\Delta\subset M\times M$.

\step{2} Then, by partition of unity, it suffices to prove the claim on sets 
of the form $U\times U\subset M\times M$. By the results in~\cite{Lesch}, in a local chart $\kappa^2:U\times U\to \kappa(U)\times \kappa(U) $ with linear coordinates $(x,y)=(x^i,y^i)_{i=1}^n$, the pseudodifferential kernel reads:
\begin{eqnarray*}
{\left(\kappa^2_*\pazocal{K}\right)}(x,x-y)=\frac{1}{(2\pi)^n}\int_{\xi\in \mathbb{R}^n} e^{i\left\langle \xi,x-y\right\rangle} \asigma(x;\xi) d^n\xi \in C^\infty(\kappa(U)\times \mathbb{R}^n\setminus\{0\})
\end{eqnarray*}
where $\asigma(x;\xi)\sim \sum_{k=0}^{+\infty} \asigma_{\cv-k}(x;\xi)$ and 
$\asigma_k\in C^\infty(\kappa(U) \times \mathbb{R}^n\setminus\{0\} ) $ 
satisfies  $\asigma_k(x;\lambda\xi)=\lambda^k\asigma(x;\xi)$, $\lambda>0$ for $\vert \xi\vert>0$.
{The normal form in Proposition ~\ref{p:normalform} yields the existence
of}
coordinate functions $(x^i,h^i)_{i=1}^n$, where $(x^i)_{i=1}^n$ are the initial linear coordinates, such that
$\kappa^2_*\euler=\sum_{i=1}^n h^i\partial_{h^i}$. We also view the coordinates $(h^i)_{i=1}^n$ as \emph{coordinate functions} $\left(h^i(x,y)\right)_{i=1}^n$
on $\kappa^2(U \times U)$, we also use the short notation $h(x,y)=(h^i(x,y))_{i=1}^n\in C^\infty(\kappa(U)^2,\mathbb{R}^n)$.
By the Kuranishi trick, the kernel
$\kappa^2_*{\pazocal{K}}$ can be rewritten as
\begin{eqnarray*}
\kappa^2_*{\pazocal{K}}(x,x-y)=\frac{1}{(2\pi)^n}\int_{\xi\in \mathbb{R}^n} e^{i\left\langle \xi,h(x,y)\right\rangle} \asigma(x; \t M(x,y)^{-1} \xi) \module{M(x,y)}^{-1}d^n\xi \\ \in C^\infty(\kappa(U)\times \mathbb{R}^n\setminus\{0\})
\end{eqnarray*}
where $\module{M(x,y)}=\det{M(x,y)}$, and the matrix $M\in C^\infty(\kappa(U)^2,\GL_n(\mathbb{R}))$ satisfies $M(x,x)=\id$, $x-y=M(x,y)h(x,y)$. Since $(x^i,y^i)_{i=1}^n$ and $(x^i,h^i)_{i=1}^n$ are both coordinates systems in $\kappa(U)\times \kappa(U)$, we can view $(x-y)=(x^i-y^i)_{i=1}^n(.,.)\in C^\infty(\kappa(U)\times \mathbb{R}^n,\mathbb{R}^n)$
as a smooth function of $(x,h)\in  \kappa(U)\times \mathbb{R}^n$
and  $M^{-1}(x,h)$ can be expressed
as an integral: 
$$
M^{-1}(x,h)=\int_0^1  d(x-y)|_{(x,th)} dt.
$$

\step{3} We need to eliminate the dependence in the $h$ variable in the symbol $A(x,y;\xi)= \asigma(x; \t M(x,y)^{-1} \xi)\module{M(x,y)}^{-1}$
keeping in mind this symbol has the polyhomogeneous expansion in the $\xi$ variable
$$
A(x,y;\xi)\sim \sum_{k=0}^{+\infty}\asigma_{\cv-k}(x; \t M(x,y)^{-1} \xi)\module{M(x,y)}^{-1}.
$$
By~\cite[Thm.~3.1]{shubin}, if we set $A(x,y;\xi)=\asigma(x; \t M(x,y)^{-1}\xi)\module{M(x,y)}^{-1}$, 
then:
\begin{eqnarray*}
A(x,y;\xi)\sim {\sum_\beta} \frac{i^{-\vert \beta\vert}}{\beta !} \partial_\xi^\beta \partial_y^\beta A(x,y;\xi) |_{x=y}
\end{eqnarray*} 
which implies that if we set $A_{\cv-k}(x,y;\xi)=\asigma_{\cv-k}(x; \t M(x,y)^{-1} \xi)\module{M(x,y)}^{-1}$, we get
the polyhomogeneous asymptotic expansion:
\begin{eqnarray}\label{e:exp1referee}
A(x,y;\xi)\sim \sum_{p=0}^{+\infty} {\sum_{\vert \beta\vert+k=p } \frac{i^{-\vert \beta\vert}}{\beta !} \partial_\xi^\beta \partial_y^\beta  A_{\cv-k}(x,y;\xi) |_{x=y} }
\end{eqnarray}
where in the sum over $p$, each term is homogeneous of degree $\cv-p$ w.r.t.~scaling in the variable $\xi$.
At this step, 
we obtain a representation of the form
\begin{eqnarray}\label{e:exp2referee}
{\left(\kappa^2_*\pazocal{K}\right)}(x,x-y)=\frac{1}{(2\pi)^n}\int_{\xi\in \mathbb{R}^n} e^{i\left\langle \xi,h(x,y)\right\rangle} \tilde{\asigma}(x;  \xi) d^n\xi \in C^\infty(\kappa(U)\times \mathbb{R}^n\setminus\{0\})
\end{eqnarray}
where $\tilde{\asigma}\in C^\infty(\kappa(U)\times \mathbb{R}^n)$ is a polyhomogeneous symbol.

\step{4} { It is at this particular step that we start to carefully distinguish between the cases $\alpha\in\mathbb{C}\setminus(\clopen{-n,+\infty}\cap \mathbb{Z})) $, which is in a certain sense easier to handle, and the case where $\alpha$ is an integer such that $\alpha\geqslant -n$.} {Up to a modification of ${\pazocal{K}}$} with a smoothing operator, we can always assume that $\tilde{\asigma}$ is smooth in $\xi$ and supported in $\vert \xi\vert\geqslant 1$. 
For every $N$, let us decompose
$$\tilde{\asigma}(x;\xi)=\sum_{k=0}^N  \tilde{\asigma}_{\cv-k}(x;\xi)+R_{\cv-N}(x;\xi) $$
where the behaviour  of the summands can be summarized as follows:
\ben
\item $R_{\cv-N}\in C^\infty\left(\kappa(U)\times \mathbb{R}^n\setminus\{0\} \right)$ and satisfies the estimate 
$$ \forall \xi\,  \text{ s.t. } \vert\xi\vert\geqslant 1, \ \forall x\in \kappa(U),\ \vert \partial_\xi^\beta R_{\cv-N}(x;\xi) \vert\leqslant C_{\cv-N,\beta} \vert\xi\vert^{{\cv-N-\beta} },  $$
and $R_{\cv-N}(x;.) $ extends as a distribution in $\kappa(U)\times \mathbb{R}^n $
of order $ N- \cv-n +1$  by \cite[Thm.~1.8]{DangAHP} since $R_{\cv-N}(x;.) $  satisfies the required weak homogeneity assumption.
\item If $\cv-k>-n$, then the symbol $ \tilde{\asigma}_{\cv-k}\in  C^\infty\left(\kappa(U)\times \mathbb{R}^n\setminus\{0\} \right)$ is homogeneous of degree $\cv-k $
and extends uniquely as a tempered distribution in $\xi$ homogeneous of degree $\cv-k$ by~\cite[Thm.~3.2.3]{H}.
\item { If $\cv-k\leqslant -n$ and
 $\cv\in\mathbb{C}\setminus(\clopen{-n,+\infty}\cap \mathbb{Z})$, 
then observe that $\cv-k\in\mathbb{C}\setminus(\clopen{-n,+\infty}\cap \mathbb{Z}) $ hence
$ \tilde{\asigma}_{\cv-k}\in  C^\infty\left(\kappa(U)\times \mathbb{R}^n\setminus\{0\} \right)$ is homogeneous of degree $\cv-k $ in $\xi$
and extends \emph{uniquely} as a tempered distribution in $\xi$ homogeneous of degree $\cv-k$ by~\cite[Thm.~3.2.4]{H}. If $\cv-k\leqslant -n$ and
$\cv\geqslant -n$ is an integer,}
then $ \tilde{\asigma}_{\cv-k}\in  C^\infty\left(\kappa(U)\times \mathbb{R}^n\setminus\{0\} \right)$ is homogeneous of degree $\cv-k $ in $\xi$
and extends \emph{non--uniquely} as a tempered distribution in $\xi$ quasihomogeneous of degree $\cv-k$ by~\cite[Thm.~3.2.4]{H}.
There are Jordan blocks in the scaling (see~\cite[(3.2.24)$^\prime$]{H}), in the sense that  we can choose the distributional extension in $C^\infty(\kappa(U),\pazocal{S}^\prime(\mathbb{R}^n))$ in such a way that:
$$ \left(\xi_i\partial_{\xi_i}-\cv+k\right)  \tilde{\asigma}_{\cv-k}= \sum_{\vert \beta\vert=k-\cv-n} C_\beta(x) \partial_\xi^\beta\delta^{\mathbb{R}^n}_{\{0\}}(\xi) .  $$
\een

\step{5} We now study the consequences of the above representation in position space. { If $\cv\geqslant -n$ is an integer then  we} have
$$
\bea
\frac{1}{(2\pi)^n}\int_{\xi\in \mathbb{R}^n} e^{i\left\langle \xi,h\right\rangle} \tilde{\asigma}(x;\xi)d^n\xi &=
\sum_{k=0}^{\cv+n-1} T_{ n+\cv-k}(x,h) + \sum_{k=\cv+n}^{N} T_{n+\cv-k}(x,h)  \fantom +
\frac{1}{(2\pi)^n}\int_{\xi\in \mathbb{R}^n} e^{i\left\langle \xi,h\right\rangle}R_{\cv-N}(x;\xi)d^n\xi,
\eea
$$
where 
$$
T_{n+\cv-k}(x,h)=\frac{1}{(2\pi)^n}\int_{\xi\in \mathbb{R}^n} e^{i\left\langle \xi,h\right\rangle} \tilde{\asigma}_{\cv-k}(x;\xi)d^n\xi.
$$
It follows that by inverse Fourier transform, when $\cv-k>-n $, $ T_{n+\cv-k}(x,.)$ is tempered in the variable $h$ and is homogeneous in the sense of tempered distributions:
$$\forall \lambda>0,  \quad T_{n+\cv-k}(x,\lambda h)=\lambda^{k-n-\cv}T_{n+\cv-k}(x,h).$$

When $\cv-k\leqslant -n $, the distribution $T_{n+\cv-k}$ is quasihomogeneous in the variable $h$, i.e., when we scale with any $\lambda>0$ w.r.t.~$h$  there is a $\log\lambda$ which appears in factor:
$$\left\langle T_{n+\cv-k}(x,\lambda.),\varphi\right\rangle=\lambda^{n- \cv+k} \left\langle T_{n+\cv-k}(x,.),\varphi\right\rangle+\lambda^{n- \cv+k}\log\lambda \left\langle (\euler-\cv+k) T_{n+\cv-k}(x,.),\varphi\right\rangle  .$$
Observe that the remainder term 
reads:
\begin{eqnarray*}
\frac{1}{(2\pi)^n}\int_{\xi\in \mathbb{R}^n} e^{i\left\langle \xi,h\right\rangle}R_{\cv-N}(x;\xi)d^n\xi
\end{eqnarray*}
which belongs to $\pazocal{C}^{N-\cv-n}$ since for $\chi\in C^\infty_\c(\mathbb{R}^n)$, $\chi=1$ near $\xi=0$, we get:
$$ \vert (1-\chi)(\xi) R_{\cv-N}(x;\xi) \vert\leqslant C_{\cv-N}(1+\vert \xi\vert)^{\cv-N}$$ 
which implies that $\int_{\xi\in \mathbb{R}^n} e^{i\left\langle \xi,h\right\rangle}(1-\chi)(\xi)R_{\cv-N}(x;\xi)d^n\xi\in \pazocal{C}^{N-\cv-n}$ by  \cite[Lem.~D.2]{Dang2020} and we can also  observe that $\int_{\xi\in \mathbb{R}^n} e^{i\left\langle \xi,h\right\rangle}\chi(\xi)R_{\cv-N}(x;\xi)d^n\xi$ is analytic in $h$ by the Paley--Wiener theorem.
{ If $\cv\in\mathbb{C}\setminus(\clopen{-n,+\infty}\cap \mathbb{Z})$, then we have a simpler decomposition
$$ \bea
\frac{1}{(2\pi)^n}\int_{\xi\in \mathbb{R}^n} e^{i\left\langle \xi,h\right\rangle} \tilde{\asigma}(x;\xi)d^n\xi &=
\sum_{k=0}^{N} T_{ n+\cv-k}(x,h) +
\frac{1}{(2\pi)^n}\int_{\xi\in \mathbb{R}^n} e^{i\left\langle \xi,h\right\rangle}R_{\cv-N}(x;\xi)d^n\xi,
\eea$$
where each $T_{ n+\cv-k}(x,h)$ is smooth in $x$ and a tempered distribution in $h$ homogeneous of degree $n+\cv-k$ (there are no logarithmic terms).
}

\step{6}  Observe that in the new coordinates $(x,h)$, 
the scaling with respect to $\euler$ takes the simple form
$ {\left( e^{-t\euler}f\right)} (x,h) = f(x,e^{-t} h)$ for smooth functions $f$.
So the provisional conclusion { for integer $\cv\geqslant -n$} is that when we scale w.r.t.~the Euler vector field, we get an asymptotic expansion 
in terms of conormal distributions: 
$$ \bea
{ e^{-t\euler}}{ \pazocal{K}}&=\sum_{k=0}^{\cv+n-1} {e^{-(n+\cv-k)t}} T_{n+\cv-k}+T_0+t{ \left( \euler T_0\right)}  \fantom +\sum_{k=\cv+n }^N 
{e^{-(n+\cv-k)t}} \left(T_{n+\cv-k}+t  (\euler -(k-\cv-n) ) T_{n+\cv-k} \right) + R(x,e^{-t}h)
\eea
$$
where $C_0=\euler T_0$ and the remainder term $R$ is a H\"older function of regularity $\pazocal{C}^{N-\cv-n}$ so it has a Taylor expansion up to order $N-\cv-n$. 
By taking the Laplace transform in the variable $t$, for any test form $\chi$, we find that the dynamical correlator 
$$ \int_0^\infty e^{-ts} \left\langle { \left(e^{-t\euler} \pazocal{K}\right)},\chi \right\rangle dt  $$
admits an analytic continuation
to a meromorphic function on  $\mathbb{C}\setminus \{n-\cv,\dots ,0 ,-1,\dots \}$
with simple poles at $\{n-\cv,\dots ,1 \}$ and poles of order at most $2$
at the points $\{0,-1,\dots, \}$.
We have a Laurent series expansion of the form:
$$\bea
 \int_0^\infty e^{-ts} { \left(e^{-t\euler} \pazocal{K}\right)} dt&=\sum_{k=0}^{\cv+n-1} \frac{ T_{n+\cv-k}}{s+k-\cv-n}+ \frac{T_0}{s}+\frac{\euler T_0}{s^2}  
 \fantom +\sum_{k=\cv+n }^N \frac{ T_{n+\cv-k}}{s+k-\cv-n}+ \frac{(\euler-k+\cv+n)T_{n+\cv-k}}{(s+k-\cv-n)^2 } \fantom + \int_0^\infty e^{-ts}  R(x,e^{-t}h) dt,
\eea
$$
where the term $\int_0^\infty e^{-ts} R(x,e^{-t}h)dt$ is holomorphic on the half-plane $\Re s >0$ and meromorphic
on the half-plane $\Re s > \cv+n-N$ due to the H\"older regularity $R\in \pazocal{C}^{N-\cv-n}$.

{ If $\cv\in\mathbb{C}\setminus(\clopen{-n,+\infty}\cap \mathbb{Z})$, then the above discussion much simplifies  because of the absence of logarithmic mixing and we find that $\int_0^\infty e^{-ts}  \left( e^{-t\euler}\pazocal{K}\right) dt$ extends as a meromorphic function with only simple poles at $n-\cv, n-1-\cv,\dots,$ and therefore $0$ is not a pole of $\mathfrak{L}_s\pazocal{K}$. It means that $\Pi_0(\pazocal{K})=0$ when $\cv\in\mathbb{C}\setminus(\clopen{-n,+\infty}\cap \mathbb{Z})$.}
\end{refproof}

\subsection{Dynamical residue equals Wodzicki residue in pseudodifferential case}\label{ss:wodzicki}

The log-polyhomogeneity of pseudodifferential Schwartz kernels ensures that their dynamical residue is well-defined. Our next objective is to show that it coincides with the Guillemin--Wodzicki residue. 

More precisely, if $\Psi^m_\ph(M)$ is the class of classical pseudodifferential operators of order $m$, we are interested in  the \emph{Guillemin--Wodzicki residue density} of $A\in\Psi^m_\ph(M)$, which can be defined at any $x\in M$ as follows. In a local coordinate chart $\kappa:U\mapsto \kappa(U)\subset \mathbb{R}^n$, the symbol $a(x;\xi)$ is given by
$$ \left(\kappa_{*} A\left(\kappa^*u\right)\right)(x)=\frac{1}{(2\pi)^n} \int_{\mathbb{R}^n\times \mathbb{R}^n} e^{i\left\langle\xi,x-y\right\rangle}\asigma(x;\xi)u(y) d^n\xi d^ny$$
 for all  $u\in C^\infty_{\rm c}\left(\kappa\left(U\right)\right)$, and  one defines the density
$$
\wres A\defeq \frac{1}{(2\pi)^n}\left(\int_{\mathbb{S}^{n-1}}  \asigma_{-n}(x;\xi) \iota_{V} d^n\xi\right)d^nx  \in \vert\Lambda^{\top}M \vert,
$$
where $V=\sum_{i=1}^n\xi_i\partial_{\xi_i}$ and  $\asigma_{-n}$ is the symbol of order $-n$ in the polyhomogeneous expansion.
If for instance $M$ is compact then the \emph{Guillemin--Wodzicki residue} is  obtained by integrating over $x$. In what follows we will only consider densities as this allows for greater generality.

{ Note that when $A\in \Psi_{\rm cl}^m(M)$ with $m\in \mathbb{C}\setminus \clopen{-n,+\infty}\cap \mathbb{Z}$
then the above residue vanishes because in this case there is no term homogeneous of degree $-n$ in the asymptotic expansion of the symbol.}

It is proved in~\cite[Prop.~4.5]{Lesch} that the residue density is intrinsic. This is related to the fact that in the local chart,  $d^nx\,d^n\xi$ is the Liouville measure, which is intrinsic and depends only on the canonical symplectic structure on $T^*M$.

\begin{thm}[Wodzicki residue, dynamical formulation]\label{wodzickipdo}
Let $M$ be a smooth manifold and let $K_{A}(.,.)\pi_2^*\dvol_g\in \pazocal{D}^\prime_{N^*\Delta}(M\times M)\otimes \pi_2^*\vert \Lambda^{\top}M \vert$ be the kernel of a classical pseudodifferential operator $A\in \Psi^\cv_{\ph}(M)$ of order {$\cv\in \mathbb{C}$}. 
Then, for every Euler vector field $\euler$ we have the identity
\beq \label{toprove}
\wres A= (\resdyn {\pazocal{K}_{A}}) \dvol_g,
\eeq
where $\wres A\in\vert\Lambda^{\top}M \vert$ is the  Guillemin--Wodzicki residue density of $A$ and $\resdyn {\pazocal{K}_{A}}=\iota_\Delta^* \euler ( \Pi_0 ({\pazocal{K}_{A}}))$ is the dynamical residue of ${\pazocal{K}_{A}}$. { If $\cv\in \mathbb{C}\setminus \clopen{-n,+\infty}\cap \mathbb{Z}$
then both sides of the above equality vanish.}
\end{thm}

In particular,  $ (\resdyn {\pazocal{K}_{A}}) \dvol_g$ {does not depend} on $\euler$.

\begin{refproof}{Theorem \ref{wodzickipdo}}
We  use the notation from the proof of Proposition~\ref{p:pseudopoly}.
Recall that 
$$ \Pi_0({\pazocal{K}_{A}})(x,h)=T_0(x,h)=\frac{1}{(2\pi)^n} \int_{\xi\in \mathbb{R}^n} e^{i\left\langle \xi,h\right\rangle}\tilde{\asigma}_{-n}(x;\xi)d^n\xi$$
where the oscillatory integral representation uses the homogeneous components of the symbol denoted by $\tilde{\asigma} \in  C^\infty(\kappa(U)\times \mathbb{R}^n)$; this
symbol $\tilde{\asigma}$ was constructed from the initial symbol $\asigma\in C^\infty(\kappa(U)\times \mathbb{R}^n)$ using the Kuranishi trick and 
is adapted to the coordinate frame $(x,h)\in C^\infty(\kappa(U)\times \mathbb{R}^n, \mathbb{R}^{2n} )$ 
in which $\euler$ has the normal form $\kappa^2_*\euler=h^i\partial_{h^i}$.
Let us examine  the meaning of the term $\euler T_0$ and
relate it to the Wodzicki residue. 
By  Proposition~\ref{p:residue}, 
the residue is the homological obstruction for the term
$\tilde{\asigma}_{-n}(x;.) $ to admit a scale invariant distributional extension to $\kappa(U)\times \mathbb{R}^n$. 
By Remark~\ref{l:scalanomal},
this reads
$$
(\xi_i\partial_{\xi_i}-n) \tilde{\asigma}_{-n}(x;\xi)= \left(\int_{\vert \xi\vert=1}\tilde{\asigma}_{-n}(x;\xi) \iota_{\sum_{i=1}^n\xi_i\partial_{\xi_i}} d^n\xi \right)\delta_{\{0\}}(\xi)
$$
{ where $\iota_{\sum_{i=1}^n\xi_i\partial_{\xi_i}}$ is the contraction operator by the vector field $\sum_{i=1}^n\xi_i\partial_{\xi_i}$ in the Cartan calculus.}
By inverse Fourier transform, $\euler T_0=\frac{1}{(2\pi)^n} \left(\int_{\vert \xi\vert=1}\tilde{\asigma}_{-n}(x;\xi)\iota_{\sum_{i=1}^n\xi_i\partial_{\xi_i}} d^n\xi\right)$, which  is a smooth function of $x\in \kappa(U)$.
We are not finished yet since the Wodzicki residue density is defined in terms of the 
symbol $\asigma(x;\xi) \in C^\infty(\kappa(U)\times \mathbb{R}^n)$ we started with. Let us recall that $\asigma$ is defined in  
such a way that $\kappa^2_*{\pazocal{K}_A}(x,x-y)=\frac{1}{(2\pi)^n} \int_{\xi\in \mathbb{R}^n} e^{i\left\langle \xi,x-y\right\rangle}\asigma(x;\xi)d^n\xi$,
and the Wodzicki residue equals $$
\wres(A)(x)=\frac{1}{(2\pi)^n} \int_{\vert \xi \vert=1}\asigma_{-n}(x;\xi) \iota_{\sum_{i=1}^n\xi_i\partial_{\xi_i}} d^n\xi.$$
{ We use the identity from equation~(\ref{e:exp1referee}): $$\tilde{a}(x;\xi)\sim A(x,y;\xi)\sim \sum_{p=0}^\infty \sum_{\vert \beta\vert+k=p} \frac{i^{-\vert\beta\vert}}{\beta !} \left(\partial_\xi^\beta\partial_y^\beta A_{\alpha-k}\right)(x,y;\xi)|_{x=y} .$$
For the residue computation, we need to extract the relevant term $\tilde{a}_{-n}$ on the r.h.s.~which is homogeneous of degree $-n$, so we need to set $\cv-p=-n$ hence $p=n+\cv $. This term reads $\tilde{a}_{-n}(x;\xi)=\sum_{\vert \beta\vert+k=n+\cv} \frac{i^{-\vert\beta\vert}}{\beta !} \left(\partial_\xi^\beta\partial_y^\beta A_{\cv-k}\right)(x,y;\xi)|_{x=y} $.}

We now make the crucial observation that for all $x\in \kappa(U)$,
$$ 
\bea
&\int_{\vert \xi\vert=1}  \tilde{\asigma}_{-n}(x;\xi) \iota_{\sum_{i=1}^n\xi_i\partial_{\xi_i}} d^n\xi
\\ &=
 \sum_{\vert \beta\vert+k={ n+\cv}  }
\int_{\vert \xi\vert=1} \frac{i^{-\vert \beta\vert}}{\beta !} \partial_\xi^\beta \partial_y^\beta  A_{\cv-k-\beta}(x,y;\xi) |_{x=y} \iota_{\sum_{i=1}^n\xi_i\partial_{\xi_i}} d^n\xi \\
&=\int_{\vert \xi\vert=1}  A_{-n}(x,y;\xi) |_{x=y} \iota_{\sum_{i=1}^n\xi_i\partial_{\xi_i}} d^n\xi=\int_{\vert \xi\vert=1}  \asigma_{-n}(x;\xi) \iota_{\sum_{i=1}^n\xi_i\partial_{\xi_i}} d^n\xi
\eea 
$$
by the vanishing property ({Corollary}~\ref{l:vanishing1}), which implies that the integral of 
all the terms with derivatives vanish.
Therefore by inverse Fourier transform, we find that
\begin{eqnarray}
C_0(x)=\frac{1}{(2\pi)^n}\int_{\vert \xi\vert=1}  \asigma_{-n}(x;\xi) \iota_{\sum_{i=1}^n\xi_i\partial_{\xi_i}} d^n\xi.
\end{eqnarray}
 The residue density $\left(\int_{\vert \xi\vert=1}  \asigma_{-n}(x;\xi)\iota_{\sum_{i=1}^n\xi_i\partial_{\xi_i}} d^n\xi\right) {d^nx}$ is \emph{intrinsic} as proved by Lesch \cite[Prop 4.5]{Lesch} (it is defined in coordinate charts but satisfies some compatibility conditions that makes it intrinsic on $M$).
To conclude observe that $\euler^2T_0=0$~\footnote{This is a consequence of the Jordan blocks having only rank $2$.} hence by the Cauchy formula, for any small disc $D$ around $0$~:
\begin{eqnarray*}
\frac{1}{2i\pi}\int_{\partial D}{ \left(\euler \pazocal{K}_A\right)}(z) dz|_{U\times U}={\left(\euler T_0\right)}(x,y)|_{U\times U}=\frac{1}{(2\pi)^n}\int_{\vert \xi\vert=1}  \asigma_{-n}(x;\xi) \iota_{\sum_{i=1}^n\xi_i\partial_{\xi_i}} d^n\xi,
\end{eqnarray*}
which in combination with the fact that $y\mapsto{\left(\euler T_0\right)}(x,y)$ is locally constant proves \eqref{toprove} on $U$. The above identity globalizes immediately, which finishes the proof.
\end{refproof}

\section{Holonomic singularities of the Hadamard parametrix} 

\label{section4}

\subsection{Hadamard parametrix} \label{s:hadamardformal} 

We now consider  the setting of a { time-oriented} Lorentzian manifold $(M,g)$, and we assume it is of \emph{even} dimension $n$. 

Let $P=\square_g$ be the wave operator (or d'Alembertian), i.e.~it is the   Laplace--Beltrami operator associated to the Lorentzian metric $g$. Explicitly, using the notation $\module{g}=\module{\det g}$,
$$
\bea
P&=\module{g(x)}^{-\frac{1}{2}}\partial_{x^j} \module{g(x)}^{\frac{1}{2}}g^{jk}(x)\partial_{x^k}  \\
& =  \partial_{x^j}g^{jk}(x)\partial_{x^k}+b^k(x)\partial_{x^k} 
\eea
$$
where we sum over repeated indices, and $b^k(x)=\module{g(x)}^{-\12} g^{jk}(x)(\partial_{x^j}\module{g(x)}^{\frac{1}{2}} )$. For $\Im z\geqslant 0$ we consider the operator $P-z$. 

The Hadamard parametrix for $P-z$ is constructed in several steps which we briefly recall following \cite{Dang2020}.

\subsubsection*{Step 1} Let $\eta=dx_0^2-(dx_1^2+\cdots+dx_{n-1}^2)$ be the  Minkowski metric on $\rr^n$, and consider the corresponding quadratic form $$
\vert\xi\vert_\eta^2 =  -\xi_0^2+\sum_{i=1}^{n-1}\xi_i^2,
$$
defined for convenience with a {minus} sign. For $\cv\in \cc$ and $\Im z >0$,  the distribution $\left(\vert\xi\vert_\eta^2-z \right)^{-\cv}$ is well-defined by pull-back from $\rr$. More generally, for $\Im z \geqslant 0$, the limit $\left(\vert\xi\vert_\eta^2-z -i0\right)^{-\cv}=\lim_{\varepsilon\to 0^+}\left(\vert\xi\vert_\eta^2-z -i\varepsilon\right)^{-\cv}$ from the upper half-plane is well defined as a distribution on $\rr^n\setminus \{0\}$. If $z\neq 0$ it can be  extended to a  family of homogeneous distributions on $\rr^n$, holomorphic in $\cv\in \cc$ (and to a meromorphic family if $z=0$).  We introduce special notation for its appropriately normalized Fourier transform,
\beq\label{eq:defFsz}
\Fs{z}\defeq\frac{\Gamma(\cv+1)}{(2\pi)^{n}} \int e^{i\left\langle x,\xi \right\rangle}\left(\vert\xi\vert_\eta^2-i0-z \right)^{-\cv-1}d^{n}\xi,
\eeq
which defines a  family of distributions on $\rr^n$, holomorphic in $\cv\in\cc\setminus\{-1,-2,\dots\}$ for $\Im z\geqslant 0$, $z\neq 0$.

\subsubsection*{Step 2}   Next, one { pulls back} the distributions $\Fse{z}$ to a neighborhood of the diagonal $\diag\subset M\times M$  using the exponential map. 

More precisely, this can be done as follows. Let  $\exp_x:T_xM\to M$ be the exponential geodesic map. We consider a neighborhood
of the zero section $\zero $ in $TM$
on which the 
map
\beq \label{xv}
(x;v)\mapsto (x,\exp_x(v))\in M^2
\eeq
is a local
diffeomorphism 
onto its image, denoted   by $\pazocal{U}$. 
Let 
{$(e_1,\dots,e_n)$} be a local { time-oriented} orthonormal  frame
defined on an open set and {$(\varalpha^i)_{i=1}^n$} the corresponding  coframe. For $(x_1,x_2)\in\pazocal{U}$ (with  $x_1,x_2$  in that open set),  we define the map
\beq\label{applipullback}
 G: (x_1,x_2)\mapsto { 
\Big(G^i(x_1,x_2)=
\underset{\in T_{x_1}^* M}{\underbrace{\varalpha^i_{x_1}}}\underset{\in T_{x_1}M}{\underbrace{(\exp_{x_1}^{-1}(x_2))}   } \Big)_{i=1}^n \in\mathbb{R}^{n}.}
\eeq
 Here, $(x_1,x_2) \mapsto (x_1;\exp_{x_1}^{-1}(x_2))$ is a diffeomorphism as it is the inverse of \eqref{xv}, and so $G$ is a submersion.  
   
For any distribution $f$ in
$\pazocal{D}^\prime(\mathbb{R}^{n})$, we can consider the pull-back $(x_1,x_2)\mapsto G^*f(x_1,x_2)$, and if $f$ is $O(1,n-1)_+^\uparrow$-invariant, 
then
the pull-back  does not depend on the choice
of orthonormal  frame $(e_\mu)_\mu$. This allows us to canonically define the pullback $G^*f\in \pazocal{D}^\prime(\cU)$, of 
$O(1,n-1)_+^\uparrow$-invariant 
distributions to distributions
defined on an open set $\pazocal{U}$ which is in fact a neighborhood 
of  the diagonal $\diag$.

\begin{defi}\label{d:fsneardiag}
For  $\cv\in\cc$, the distribution 
$\Fe{z}= G^*\Fse{z}\in \pazocal{D}^\prime(\pazocal{U})$ is defined by pull-back of the $O(1,n-1)_+^\uparrow$-invariant distribution  $\Fse{z}\in \pazocal{D}^\prime\left(\mathbb{R}^{n}\right)$  introduced  in  \eqref{eq:defFsz}.
\end{defi}

\subsubsection*{Step 3}   The Hadamard parametrix is  constructed in normal charts using the family $\Fe{z}$. Namely, for fixed
   $\varvarm\in M$,  we  express the distribution $x\mapsto \mathbf{F}_\cv(z,\varvarm,x)$
in normal coordinates centered at $\varvarm$, defined on some $U\subset T_{x_0}{M}$. Instead of using the somewhat heavy notation $ \mathbf{F}_\cv(z,\varvarm,\exp_\varvarm(\cdot))$ we will simply write
$\Feg{z}\in \pazocal{D}^\prime(U)$. One then looks for a parametrix $H_N(z)$ of order $N$ of the form
\begin{eqnarray}
\boxed{H_N(z)=\sum_{k=0}^N u_k \Feg[k]{z} \in \pazocal{D}^\prime(U) },
\end{eqnarray}
and after computing $\left(P-z\right)H_N(z,.)$  one finds that the sequence of functions 
$(u_k)_{k=0}^\infty$ in $C^\infty(U)$ should solve the hierarchy of transport equations
\begin{eqnarray}
\boxed{2k u_k+ b^i(x)\eta_{ij}x^j u_k+2 x^i\partial_{x^i} u_k+2Pu_{k-1}=0}
\end{eqnarray}
with initial condition $u_0(0)=1$, where by convention $u_{k-1}=0$ for  $k=0$, and we sum over repeated indices. The transport equations have indeed a unique solution, and they imply  that on $U$, $H_N(z,.)$ solves
\begin{equation}
\left(P-z\right)H_N(z,.)={\module{g}^{-\frac{1}{2}}}\delta_0+(Pu_N)\mathbf{F}_N.
\end{equation}

\subsubsection*{Step 4} The final step consists in considering the dependence on $x_0$ to obtain a parametrix on the neighborhood $\cU$ of the diagonal.  One shows that $\pazocal{U} \ni (x_1,x_2) \mapsto u_k(\varalpha(\exp_{x_1}^{-1}(x_2))) $ is smooth in both arguments, and since $\mathbf{F}_{\cv}(z,.)$ is already defined on $\pazocal{U}$,  $$H_N(z,x_1,x_2)=\sum_{k=0}^N u_k(\varalpha(\exp_{x_1}^{-1}(x_2)))\mathbf{F}_\cv(z,x_1,x_2)
$$ 
is well defined as a distribution on $\pazocal{U}$.
Dropping the exponential map in the notation from now on for simplicity, the \emph{Hadamard parametrix} $H_N(z,.)$ of order $N$ is by definition the distribution
\begin{equation}
\boxed{H_N(z,.)=\sum_{k=0}^N u_k \Fe[k]{z}\in \pazocal{D}^\prime(\pazocal{U}).}
\end{equation}
Finally, we use an arbitrary cutoff function $\chi\in \cf(M^2)$ supported in $\pazocal{U}$ to extend the definition of $H_N(z,.)$ to $M^2$,
$$
\boxed{H_N(z,.)=\sum_{k=0}^N \chi u_k \Fe[k]{z}\in \pazocal{D}^\prime({M\times M}).}
$$
The Hadamard parametrix extended to $M^2$  satisfies 
\beq\label{eq:PzHN}
\left(P-z\right)  H_N(z,.)={ \module{g}^{-\frac{1}{2}}}\delta_{\diag}+(Pu_N)\mathbf{F}_N(z,.)\chi+r_N(z,.),
\eeq
where 
 $\module{g}^{-\frac{1}{2}}\delta_\Delta(x_1,x_2)$ is the Schwartz kernel of the identity map and $r_N(z,.)\in  \pazocal{D}^\prime({M\times M})$ is an error term supported in a punctured neighborhood of $\Delta$ which is due to the presence of the cutoff $\chi$.

\subsection{Oscillatory integral representation and log-polyhomogeneity}\label{ss:polhom} Given an Euler vector field $\euler$, our current objective is to  study the behaviour 
of ${\left( e^{-t\euler}H_N\right)}(z)$ and in particular  to  prove that
$H_N(z)$ is tame log-polyhomogeneous near $\Delta$. The proof uses an
oscillatory integral representation of the Hadamard parametrix involving
symbols with values in distributions { whose wave front set in the $\xi$ variable is contained in the conormal of the cone $\{Q=0\}\subset \mathbb{R}^n$. This conormal is a non-smooth Lagrangian in $T^*\mathbb{R}^n$ whose singularity is at the vertex of the cone $\{Q=0\}\subset \mathbb{R}^n$}.

{ 
\begin{rema}[Coordinate frames versus charts]
In the present part, instead of using charts we  
favor a presentation using coordinate frames
which makes notation
simpler. The two viewpoints are equivalent
since given a chart $\kappa:U\to \kappa(U)\subset \mathbb{R}^n$, the linear 
coordinates $(x^i)_{i=1}^n \in \mathbb{R}^{n*}$ on $\mathbb{R}^n$ can be pulled back on 
$U$ as a coordinate frame $(\kappa^*x^i)_{i=1}^n\in C^\infty(U;\mathbb{R}^n)$.  
\end{rema}}

We start by representing the distributions $\mathbf{F}_\cv$ defined  in \sec{s:hadamardformal} by oscillatory integrals using the 
coordinate frames from Proposition~\ref{p:normalform} adapted to 
the Euler vector field $\euler$.

\begin{lemm}\label{l:kuranishi} 
Let $(M,g)$ be a  { time-oriented}  Lorentzian manifold and $\euler$  an Euler vector field. Let $p\in \Delta$, and let  $(x^i,h^i)_{i=1}^n$ be a local coordinate frame defined on a neighborhood $\Omega\subset M\times M$ of $p$ such that $\euler=\sum_{i=1}^n h^i\partial_{h^i}$ on $\Omega$.
In this coordinate frame, $\mathbf{F}_\cv(z,.,.)$ has the representation
$$
\mathbf{F}_\cv(z,x,h)=\int_{\mathbb{R}^n} e^{i\left\langle\xi,h\right\rangle} A_\cv(z,x,h;\xi) d^n\xi,
$$
where $A_\cv$ depends holomorphically in $z\in \{\Im z>0\}$,  is homogeneous in $(z,\xi)$ of degree $-2(\cv+1)$
w.r.t.~the scaling 
$(\lambda^2z,\lambda\xi) $, and  for $\mu\neq 0$, $A_\cv(i0+\mu,.,.;.)$ is a {distribution} in $\Omega\times \mathbb{R}^{n*}$. 
\end{lemm}  

 Integrands such as $A_\cv(i0+\mu,.,.;.)$ are sometimes called distribution-valued amplitudes in the literature since they are not smooth symbols but distributions, yet they behave like symbols of oscillatory integrals in the sense they have homogeneity with respect to scaling and the scaling degree in $\xi$ is responsible for the singularities of $\mathbf{F}_\cv$.

\begin{refproof}{Lemma \ref{l:kuranishi}}
Our proof  uses in an essential way the so-called Kuranishi trick { again}. 
Let $s=(s^i)_{i=1}^n$ denote the orthonormal moving {coframe} from
\sec{s:hadamardformal}.
We denote by $\exp_{m}:T_mM\to M$ the geodesic exponential 
map induced by the metric $g$. 
We claim that
\beq\label{eq:idM}
{ \Big(  \underset{\in T^*_{(x,0)}\Omega}{ \underbrace{s^i_{(x,0)}}} 
\underset{\in T_{(x,0)}\Omega}{\underbrace{\left(\exp^{-1}_{(x,0)}(x,h)\right)}}= M(x,h)^i_jh^j\Big)_{i=1}^n }
\eeq
where 
{
$M:\Omega \ni (x,h)\mapsto (M(x,h)_i^j)_{1\leqslant i,j\leqslant n}\in \GL_n(\mathbb{R})$} is a smooth map
such that $M(x,0)=\id$.
By the fundamental theorem of calculus,
\begin{eqnarray*}
\exp^{-1}_{(x,0)}(x,h)=\int_{0}^1 \frac{d}{dt} \exp^{-1}_{(x,0)}(x,th) dt =\left(\int_{0}^1 d\exp^{-1}_{(x,0)}(x,th) dt\right)(h). 
\end{eqnarray*}
If we set $M(x,h)=s_{(x,0)}\left(\int_{0}^1 d\exp^{-1}_{(x,0)}(x,th) dt\right)$ then $M(x,0)=\id$
so up to choosing some smaller open set $\Omega$, the matrix $M(x,h)$ is
invertible for $(x,h)\in \Omega$ and satisfies \eqref{eq:idM}.

We now insert \eqref{eq:idM} into  the definition of $\mathbf{F}_\cv$:
\[\bea
\mathbf{F}_\cv(z,x,h)&=\frac{\Gamma(\cv+1)}{(2\pi)^n}\int_{\mathbb{R}^n} e^{i\left\langle\xi,s_{(x,0)}\left(\exp^{-1}_{(x,0)}(x,h)\right)\right\rangle} \left(Q(\xi)-z \right)^{-\cv-1}  d^n\xi\\
&=\frac{\Gamma(\cv+1)}{(2\pi)^n}\int_{\mathbb{R}^n} e^{i\left\langle \t M(x,h)\xi,h\right\rangle} \left(Q(\xi)-z \right)^{-\cv-1}  d^n\xi
\\
&=\frac{\Gamma(\cv+1)}{(2\pi)^n}\int_{\mathbb{R}^n} e^{i\left\langle \xi,h\right\rangle} \left(Q((\t M(x,h))^{-1}\xi)-z \right)^{-\cv-1} \module{M(x,h)}^{-1} d^n\xi.
\eea
\]
This motivates setting 
$A_\cv(z,x,h;\xi)=\left(Q((\t M(x,h))^{-1}\xi)-z \right)^{-\cv-1} \module{M(x,h)}^{-1}$ 
in $\Omega\times \mathbb{R}^{n*}$, which is homogeneous of degree $-2(\cv+1)$ w.r.t.~the scaling defined as $(\lambda^2z,\lambda\xi)$ for $\lambda>0$. If we let $\Im z\rightarrow 0^+$, then we  view $A_\cv(-m^2+i0,x,h;\xi)$ as a distribution-valued symbol defined by the pull-back of
$\left(Q(.)+m^2-i0\right)^{-\cv-1}$ by the submersive map
$\Omega\times \mathbb{R}^{n*} \ni (x,h;\xi) \mapsto (\t M(x,h))^{-1}\xi\in \mathbb{R}^{n*}$, where the fact that it is a submersion comes from the invertibility of $M(x,h)\in M_n(\mathbb{R})$  for all $(x,h)\in \Omega$.

The formal change of variable can be justified with a dyadic partition of unity $1=\chi(\xi)+\sum_{j=1}^\infty \beta(2^{-j}\xi)$ as follows.
Observe that $1=\chi((\t M(x,h))^{-1}\xi)+\sum_{j=1}^\infty \beta((\t M(x,h))^{-1}2^{-j}\xi)$. We know that
$\left(Q(\xi)-z \right)^{-\cv-1}$ is a distribution of order $\plancher{\Re \cv}+1$ hence by the change of variable formula for 
distributions:
$$ 
\bea
&\sum_{j=1}^\infty\left\langle (Q(.)-z)^{-\cv-1},\beta(2^{-j}.) e^{i\left\langle \t M(x,h).,h\right\rangle} \right\rangle+\left\langle (Q(.)-z)^{-\cv-1},\chi(.) e^{i\left\langle \t M(x,h).,h\right\rangle} \right\rangle \\
&=\sum_{j=1}^\infty\left\langle (Q((\t M(x,h))^{-1}.)-z)^{-k-1},\beta((\t M(x,h))^{-1}2^{-j}.) e^{i\left\langle .,h\right\rangle}  \right\rangle  \module{M(x,h)}^{-1}
 \fantom + \left\langle (Q((\t M(x,h))^{-1}.)-z)^{-k-1},\chi((\t M(x,h))^{-1}.) e^{i\left\langle .,h\right\rangle}  \right\rangle  \module{M(x,h)}^{-1}\\
&=\sum_{j=1}^\infty 2^{j(n-2(\cv+1))}\left\langle (Q((\t M(x,h))^{-1}.)-2^{-2j}z)^{-k-1},\beta((\t M(x,h))^{-1}.) e^{i\left\langle 2^j.,h\right\rangle}\right\rangle  \module{M(x,h)}^{-1}
\fantom + \left\langle (Q((\t M(x,h))^{-1}.)-z)^{-k-1},\chi((\t M(x,h))^{-1}.) e^{i\left\langle 2^j.,h\right\rangle}\right\rangle  \module{M(x,h)}^{-1}
\eea
$$
where the series satisfies a bound of the form
$$
\bea
&\sum_{j=1}^\infty \Big| \left\langle (Q-z)^{-\cv-1},\beta(2^{-j}.) e^{i\left\langle \t M(x,h).,h\right\rangle} \right\rangle\Big| \\ &\leqslant C \sum_{j=1}^\infty 2^{j(n-(\Re\cv+1))} \sup_{(x,h)\in \Omega} \Vert \beta((\t M(x,h))^{-1}.)\Vert_{C^{\plancher{\Re\cv}+1}}, 
\eea
$$
where $C$ does not depend on $(x,h)\in \Omega$ and the series converges absolutely for $\Re\cv$ large enough.
Then the change of variable is justified for all $\cv\in \mathbb{C}$ by analytic continuation in $\cv\in \mathbb{C}$.
\end{refproof}
%

 Given an Euler vector field $\euler$, let $(x,h)$ be the local coordinate
frame for which $\euler=h^i\partial_{h^i}$.  From the proof of Lemma  \ref{l:kuranishi}  it follows that for any sufficiently small open set $\Omega$, we can represent the Hadamard parametrix in the form
\begin{eqnarray*}
H_N(z,x,h)|_\Omega= \sum_{k=0}^N \int_{\mathbb{R}^n}  e^{i\left\langle\xi,h\right\rangle} B_{2(k+1)}(z,x,h;\xi) d^n\xi
\end{eqnarray*}
where  $B_{2(k+1)}\in \pazocal{D}^\prime(\Omega\times \mathbb{R}^{n*})$ is given by
\begin{eqnarray}\label{Bksymbol}
B_{2(k+1)}(z,x,h;\xi)=\frac{\Gamma(k+1)}{(2\pi)^{n}}\chi u_k(x,h) \left(Q((\t M(x,h))^{-1}\xi)-z \right)^{-k-1} \module{M(x,h)}^{-1},
\end{eqnarray}    
where $M(x,h)$ is the matrix satisfying \eqref{eq:idM}. Observe that $B_{2(k+1)}$ is homogeneous of degree $-2k-2$ w.r.t.~the scaling 
$ (\xi,z) \mapsto (\lambda\xi,\lambda^2z)$.


Since the Euler vector field $\euler$ reads
$\euler=h^i\partial_{h^i}$ in our local coordinates, the scaling of the Hadamard parametrix reads
$$\bea
 {\left( e^{-t\euler} H_N\right)}(z,x,h)&=H_N(z,x,e^{-t}h)=
 \sum_{k=1}^N \int_{\mathbb{R}^n}  e^{i\left\langle\xi,e^{-t}h\right\rangle} B_{2(k+1)}(z,x,e^{-t}h;\xi) d^n\xi\\
 &=\sum_{k=1}^N e^{tn} \int_{\mathbb{R}^n}  e^{i\left\langle\xi,h\right\rangle} B_{2(k+1)}(z,x,e^{-t}h;e^{t}\xi) d^n\xi.
 \eea
$$
In consequence, to capture the $t\to+\infty$ behaviour we need to compute the asymptotic expansion of
each term $ B_{2(k+1)}(z,x,\lambda h;\frac{\xi}{\lambda})$, 
and thus of $\big(Q(\t M(x,\lambda h)^{-1}\frac{\xi}{\lambda})-z\big)^{-k-1}$
as $\lambda\rightarrow 0^+$. We will see that this asymptotic expansion 
occurs in a space of holonomic distributions singular along 
the \emph{singular Lagrangian} { (it is the conormal bundle of the cone $\{Q=0\}$ in $\xi$ variables)} $$\{(\xi;\tau dQ(\xi)) \st \tau<0, Q(\xi)=0 \}.$$

\subsubsection{Asymptotic expansions of $\mathbf{F}_k(z)$ and $(Q(\frac{\xi}{\lambda})-z)^{-k-1}$.} As already remarked, 
the distribution $$(Q( \t M^{-1}(x,h) \xi)-z)^{-\cv-1}$$ is homogeneous w.r.t.~scaling
$(x,z)\mapsto(\lambda\xi,\lambda^2z) $.  We want to give a log-polyhomogeneous expansion as an asymptotic series of 
distributions in the $\xi$ variables even though $\Im z>0$.
This leads us to consider the regularized distributions
$\pf (Q(\xi)-i0)^{-k}$ and $\pf (Q(\xi)-i0)^{-k} (Q(\xi)-z)^{-1}$ for all integers $k\geqslant \frac{n}{2}$, defined as follows. 

Recall that {$(Q(\xi)-i0)^{-\cv}$} (resp.~$ (Q(\xi)-i0)^{-\cv} (Q(\xi)-z)^{-1} $ when $\Im z>0$) is a meromorphic family of tempered distributions with simple poles at  $\cv=\{\frac{n}{2},\frac{n}{2}+1,\dots \}$.  The residues are distributions supported at $\{0\}\subset \mathbb{R}^n$.

\begin{defi}\label{d:finitepart}
We define $\pf(Q(\xi)-i0)^{-k}$ {(resp.~$ \pf (Q(\xi)-i0)^{-k} (Q(\xi)-z)^{-1} $)} as the value at $\cv=k$ of the holomorphic part of the Laurent series expansion of
$(Q(\xi)-i0)^{-\cv}$ (resp  $  (Q(\xi)-i0)^{-\cv} (Q(\xi)-z)^{-1} $) near $\cv=k$. 
\end{defi} 

By application of the pull-back theorem, we immediately find that
the distribution $\pf (Q(\xi)-i0)^{-k}$ is a tempered distribution
whose wavefront set is contained in the 
\emph{singular Lagrangian} $$\{(x;\tau dQ(x)) \st Q(x)=0, \ \tau < 0 \}\cup T^*_0\mathbb{R}^n.$$ Let us briefly recall the reason why 
$\pf(Q(\xi)-i0)^{-k}$ is quasihomogeneous and 
give the equation it satisfies.
\begin{lemm}[Quasihomogeneity]
Let $\Feuler=\sum_{i=1}^n\xi_i\frac{\partial}{\partial \xi_i}$. We have the identity
$$
{\Feuler} \pf(Q(\xi)-i0)^{-k}=-2k\pf(Q(\xi)-i0)^{-k}+\res_{\cv=k}(Q(\xi)-i0)^{-\cv}
$$
and $${\Feuler}( \res_{\cv=k}(Q(\xi)-i0)^{-\cv})=-2k\res_{\cv=k}(Q(\xi)-i0)^{-\cv}.$$  Moreover, the distribution $\res_{\cv=k}(Q(\xi)-i0)^{-\cv}$ is supported at $\{0\}$.
\end{lemm}

\begin{proof}
 For non-integer $\cv$, we always have
\begin{equation}\label{e:complexhom}
{\Feuler}(Q(\xi)-i0)^{-\cv}=-2\cv (Q(\xi)-i0)^{-\cv} 
\end{equation}
since this holds true for large $-\Re\cv$  and 
extends by analytic continuation in $\cv$.

Now for $\cv$ near $k$, we use the Laurent series expansion in $\cv$ near $k$ and identifying the regular parts on both sides of \eqref{e:complexhom} yields the result.
\end{proof}

We introduce the following notation on the inverse Fourier transform side.

\begin{defi}
{ Using  Definition~\ref{d:finitepart} for the notion of finite part $\pf$, we} define
$$
\pf F_k(+i0,.):= \frac{\Gamma(k+1)}{(2\pi)^n} \int_{\mathbb{R}^n} e^{i \left\langle \xi,. \right\rangle} \pf(Q(\xi)-i0)^{-k-1} d^n\xi.
$$
\end{defi}

We now state  the main proposition of the present paragraph,
which yields asymptotic expansions for the distributions
$F_k(z,.)$. 

\begin{prop}[log-polyhomogeneity of $F_k(z,.) $]
\label{prop:logpolyF}
For every $N$, we have 
the identity
$$
(Q(\xi)-z)^{-k-1}=\sum_{p=0}^N \begin{pmatrix}
-k-1\\
p,-k-1-p
\end{pmatrix} (-1)^pz^{p} (Q(\xi)-i0)^{-(k+p+1)}+E^{\geqslant k+N+2 }+T_N(z)
$$
where $E^{\geqslant N+2+k }$ denotes the space of all distributions $T\in \pazocal{S}^\prime(\rr^n)$ such that 
$$ \lambda^{-N-2-k} T(\lambda^{-1}.)_{\lambda\in \opencl{0,1}} \mbox{ is bounded in  } \pazocal{S}^\prime(\rr^n),$$ and $T_N(z)$ is a distribution supported at $0$ depending holomorphically in $z\in \{\Im z >0\}$. 

It follows by inverse Fourier transform that
\begin{eqnarray}\label{e:expansionF}
F_k(z,.)=\sum_{p=0}^{N} \frac{(-1)^pz^{p}}{p!} {\pf F_{k+p}(+i0,.)}
+E^{\geqslant k+N+2-n }+P_N(z)
\end{eqnarray} 
where $P_N(z)$ is a polynomial function on $\mathbb{R}^n$ depending holomorphically on $z\in \{\Im z>0\}$,
hence each distribution $F_k(z,.)$ is log-polyhomogeneous.   
\end{prop}

\begin{proof}
We work in Fourier space with 
the function $\left(Q(\xi)-z\right)^{-1}$ for $\Im z>0$.
In fact, even though $\left(Q(\xi)-z\right)^{-1}$ is  a function, its asymptotic expansion in $\xi$ will involve the quasihomogeneous distributions
$\pf(Q(\xi)-i0)^{-k}$ 
because we need to consider the distributional extension to $\mathbb{R}^n$.

We start from the expression:
\begin{eqnarray*}
\sum_{k=0}^{N-1} z^k\pf \left(Q(\xi)-i0\right)^{-k-1}+z^N\pf\left(Q(\xi)-i0\right)^{-N} \left(Q(\xi)-z\right)^{-1}
\end{eqnarray*}
which is a well-defined
distribution in $\pazocal{S}^\prime(\mathbb{R}^n)$.
The product $\left(Q(\xi)-i0\right)^{-N} \left(Q(\xi)-z\right)^{-1} \in \pazocal{D}^\prime(\mathbb{R}^n\setminus\{0\}) $ is weakly homogeneous of degree $\leqslant -N-1 $ therefore it admits a distributional extension
$\pf\left(\left(Q(\xi)-i0\right)^{-N} \left(Q(\xi)-z\right)^{-1}\right)$ which is weakly homogeneous of degree $<-N-1$ and is defined by extending the distribution $\left(Q(\xi)-i0\right)^{-N} \left(Q(\xi)-z\right)^{-1}\in \pazocal{D}^\prime(\mathbb{R}^n\setminus\{0\})$ to $\pazocal{D}^\prime(\mathbb{R}^n)$, see \cite[Thm.~1.7]{DangAHP} (cf.~\cite{Meyer}).

We easily verify that we have the identity for $\Im z>0$:
\begin{eqnarray*}
(Q(\xi)-z) \left(\sum_{k=0}^{N-1} z^k\pf \left(Q(\xi)-i0\right)^{-k-1}+z^N\left(Q(\xi)-i0\right)^{-N} \left(Q(\xi)-z\right)^{-1}\right)=1
\end{eqnarray*}
in the sense of distributions on $\mathbb{R}^n\setminus \{0\}$ (we used the key fact that $Q(\xi)(Q(\xi)-i0)^{-k}=(Q(\xi)-i0)^{-k+1}$ which holds true in the distribution sense in $\pazocal{D}^\prime(\mathbb{R}^n\setminus\{0\}$). 
{Since} the term inside the large brackets above makes sense as a distribution on $\mathbb{R}^n$, 
it follows that we have the identity
$$
(Q(\xi)-z)\left(\sum_{k=0}^{N-1} z^k\pf \left(Q(\xi)-i0\right)^{-k-1}+z^N\pf\left(Q(\xi)-i0\right)^{-N} \left(Q(\xi)-z\right)^{-1}\right)=1+T_N(z) 
$$
in the sense of tempered distributions in $\pazocal{S}^\prime(\mathbb{R}^n)$, 
where $T_N(z)$ is a distribution supported at $\{0\}$ depending holomorphically on $z\in \{\Im z>0\}$.
It follows by inverse Fourier transform
that we get:
\begin{eqnarray*}
F_{0}(z,\vert x\vert_\eta)=\sum_{k=0}^{N-1} z^k \pf F_{k}(+i0,x)+E^{\geqslant N+1-n}+\pazocal{F}^{-1}\left(T_N\right)(x),
\end{eqnarray*}
where the inverse Fourier transform 
$\pazocal{F}^{-1}\left(T_N\right)(x)$ is a polynomial function in $x$.
More generally, by the same method we find that
$$
(Q(\xi)-z)^{-k}=\sum_{p=0}^N \begin{pmatrix}
-k\\
p,-k-p
\end{pmatrix} (-1)^pz^{p} {\pf (Q(\xi)-i0)^{-(k+p)}}+E^{\geqslant k+N+1 }+T_N(z)\in \pazocal{S}^\prime(\mathbb{R}^n)
$$
where the generalized binomial coefficients are defined using the
Euler $\Gamma$ function, $E^{\geqslant N+1+k } $ denotes distributions $T\in \pazocal{S}^\prime$ 
s.t. the family $ \lambda^{-N-1-k} T(\lambda^{-1}.)_{\lambda\in \opencl{0,1}}$ is bounded in $\pazocal{S}^\prime$ and $T_N(z)$ is a distribution supported at $0$ depending holomorphically in $z\in \{\Im z>0\}$. Therefore, \eqref{e:expansionF}
follows by inverse Fourier transform.
\end{proof}

We now prove that $H_N(z)\in \pazocal{D}^\prime_\Lambda(M\times M)$   is tame log-polyhomogeneous regardless of the choice of Euler vector field $\euler$.

\begin{prop}\label{prop:feynmanlog}
Let $H_N(z)$ be the Hadamard parametrix of order $N$. Then for any Euler vector field $\euler$, there exists an $\euler$-stable neighborhood  $\pazocal{U}$ of $\Delta\subset M\times M$ such that $H_N(z)\in \pazocal{D}^\prime(\pazocal{U})$
is tame log-polyhomogeneous w.r.t.~scaling with $\euler$. 
In particular,
$$
\mathfrak{L}_s H_N(z)= \int_0^\infty e^{-t(\euler+s)}H_N(z) dt \in \pazocal{D}^\prime(\pazocal{U})
$$
is a well-defined distribution and extends as a \emph{meromorphic function} of $s\in \mathbb{C}$
with poles at {$s\in -2+n-\mathbb{N}$}. The poles at {$s=k\in \mathbb{Z}$} are simple when $k<0$ and of multiplicity {at most} $2$ when $k\geqslant 0$.
\end{prop}

In the proof we will frequently make use of  smooth functions with values in tempered distributions in the following sense. 

\begin{defi}\label{d:smoothfunctionsvaluedindistributions}
If $\Omega \subset M$ is an open set, we denote by $C^\infty(\Omega)\otimes \pazocal{S}^\prime(\mathbb{R}^n)$ the space of all $U\in \pazocal{D}^\prime(\Omega\times \mathbb{R}^n)$ such that 
for all $\varphi_1\in C^\infty_{\rm c}(\Omega)$, $\varphi_2\in \pazocal{S}(\mathbb{R}^n)$,
$$
\left\langle U,\varphi_1\otimes \varphi_2 \right\rangle_{\Omega\times \mathbb{R}^n}=
\int_\Omega \left\langle U(x,.),\varphi_2\right\rangle_{\mathbb{R}^n} \varphi_1(x)\dvol_g(x)
$$
where $\Omega \ni x\mapsto  \left\langle U(x,.),\varphi_2\right\rangle_{\mathbb{R}^n}$ is $C^\infty$. 
\end{defi}

\begin{refproof}{Proposition \ref{prop:feynmanlog}} We employ a three steps asymptotic expansion. The first one comes from the Hadamard expansion, which is of the form
\begin{eqnarray*}
\sum_{k=0}^N \int_{\mathbb{R}^n} \dots e^{i\left\langle\xi,h\right\rangle} (Q(\t M(x,h)^{-1}\xi)-z)^{-k-1} \dots d^n\xi +\text{highly regular term }.
\end{eqnarray*}

\step{1} {(First expansion, in $z$).} \ The idea is to study the asymptotics
of
$$(Q(\t M(x,\lambda h)^{-1}\lambda^{-1}\xi)-z)^{-k-1}$$ when $\lambda\rightarrow 0^+$.
We start from the function
$(Q(\t M(x,h)^{-1}\xi)-{z})^{-k-1} $ where $M$ is the invertible matrix 
depending smoothly on $(x,h)$ which was obtained by the Kuranishi trick.
Then each term $(Q(\t M(x,h)^{-1}\xi)-z)^{-k-1} $
appearing in the sum is expanded in powers of $z$ times homogeneous terms in $\xi$ 
{thanks to Proposition~\ref{prop:logpolyF}}.
The expansion in powers of $z$ reads: 
$$
\bea
(Q(\t M(x, h)^{-1}\xi)-z)^{-k-1}
=\sum_{p=0}^N {(-1)^p} z^p \begin{pmatrix}
-k-1\\
p,-k-1-p
\end{pmatrix} 
\pf(Q(\t M(x, h)^{-1}\xi)-i0)^{-k-1-p} \\ +R_N(z,x,h;\xi)
\eea
$$
where $R_N(z,x,h;\xi)\in C^\infty(\Omega )\otimes \pazocal{S}^\prime(\mathbb{R}^n)$  is weakly homogeneous of degree $\geqslant -k-1-N$ in $\xi$, i.e.
\begin{eqnarray*}
 \lambda^{-N-k-1}R_N(z,x,h;\lambda^{-1}.)_{\lambda\in \opencl{0,1}} \text{ is bounded in }\pazocal{S}^\prime(\mathbb{R}^n)
\end{eqnarray*}
uniformly in $(x,h)\in K\subset \Omega$ where $K$ is a compact set.   

\step{2}   {(Second expansion, in $h$).} \ The key idea is to note that $\pf(Q(\t M(x, h)^{-1}\xi)-i0)^{-k-1-p}\in C^\infty(\Omega )\otimes \pazocal{S}^\prime(\mathbb{R}^n)$ since it is the pull-back of $\pf(Q(\xi)-i0)^{-k-1-p}\in \pazocal{S}^\prime(\mathbb{R}^n)$ by the submersion
$\Omega\times \mathbb{R}^{n*} \ni (x,h;\xi) \mapsto (\t M(x,h))^{-1}\xi\in \mathbb{R}^{n*}$. So by the push-forward theorem, for any test function $\chi\in \pazocal{S}(\mathbb{R}^n)$, the wave front set of 
$$
(x,h)\in \Omega\mapsto \left\langle \pf(Q(\t M(x, h)^{-1}.)-i0)^{-k-1-p},\chi\right\rangle $$ is empty which implies $\pf(Q(\t M(x, h)^{-1}\xi)-i0)^{-k-1-p}\in C^\infty(\Omega )\otimes \pazocal{S}^\prime(\mathbb{R}^n)$.
The important subtlety is that when we differentiate $(Q(\t M(x,h)^{-1}\xi)-i0)^{-k} $ in $(x,h)$, we lose 
distributional order in $\xi$. This is why we are not in usual spaces of symbols where differentiating in $(x,h)$ does not affect the regularity in $\xi$. However, all the $(x,h)$ derivatives $D_{x,h}^\beta(Q(\t M(x,h)^{-1}\xi)-i0)^{-k-1} $ are quasihomogeneous in $\xi$ of degree $-2k-2$:
$$D_{x,h}^\beta\big(Q(\t M(x,h)^{-1}\lambda^{-1}\xi )-i0\big)^{-k-1}=\lambda^{2k+2}D_{x,h}^\beta\big(Q(\t M(x,h)^{-1}\xi)-i0\big)^{-k}  .$$
We then expand each term
$\pf(Q(\t M(x, h)^{-1}\xi)-i0)^{-k-p-1}$ using a Taylor expansion with remainder in the variable $h$ combined with the Fa\`a di Bruno formula { (which serves to compute higher derivatives of the composition of two functions)}. {Applying the Fa\`a di Bruno formula in our particular case, we get} for all $\cv$,
$$
 \pf(Q(\t M(x, h)^{-1}\xi)-i0)^{-\cv} =\sum_{\ell, \vert \beta_1\vert+\dots+\vert\beta_\ell\vert\leqslant N} h^\beta Q_\beta( x,h;\xi)
|_{(x,0)}+I_N(z,x,h;\xi).
$$
where we denoted
$$
\bea
Q_\beta( x,h;\xi) & = \frac{(-\cv)\dots(-\cv-\ell-1) \left(\partial_h^{\beta_1}Q(\t M^{-1}(x,h)\xi)\right)\dots
\left(\partial_h^{\beta_\ell}Q(\t M^{-1}(x,h)\xi)\right) }{\beta_1!\dots\beta_\ell!\ell!} 
 \fantom  \times \pf(Q(\xi)-i0)^{-\cv-\ell}.
\eea
$$

Each  $h^\beta Q_\beta( x,h;\xi) |_{(x,0)}$
term is \emph{polynomial} in $h$ and a distribution in $\xi$ homogeneous of degree $-2\cv$ of order $\plancher{\Re\cv}+\ell+1$.
Let us describe the integral remainder,
\begin{eqnarray*}
I_N(z,x,h;\xi)= \sum_{\vert\beta\vert=N+1} \frac{(N+1)h^\beta}{\beta !} \left(\int_0^1 (1-s)^{N}\partial^\beta_h\pf(Q(\t M(x,sh)^{-1}\xi)-i0)^{-\cv}ds\right) 
\end{eqnarray*}
where the derivative
$\partial^\beta_h\pf(Q(\t M(x,sh)^{-1}\xi)-i0)^{-\cv}$ can  be expanded by Fa\`a di Bruno formula as above. 
We deduce that the term $\partial^\beta_h\pf(Q(\t M(x,sh)^{-1}\xi)-i0)^{-\cv}$ 
is continuous in both $(s,h)$ with values in distributions in $\xi$ quasihomogeneous of degree $-2\cv$
of order $\plancher{\Re\cv}+N+2$ uniformly in $(x,sh)$. Therefore 
$I_N(z,x,h;\xi)$ is continuous in $(x,h)$ with values in distributions in $\xi$ quasihomogeneous of degree $-2\cv$
of order $\plancher{\Re\cv}+N+2$ uniformly in $(x,h)$. 

\step{3} {(Combination of both expansions).} \ Combining both expansions yields an expansion of 
$$\big(Q(\t M(x,\lambda h)^{-1}\lambda^{-1}\xi)-z\big)^{-k-1} $$ in powers of $z$ and of $h$
with remainder that we  write shortly as:
$$
\bea
&(Q(\t M(x,h)^{-1}\xi)-z)^{-k-1}
\\ & = \sum_{\ell,\sum_{i=1}^\ell\vert \beta_i\vert+2k+2+2p\leqslant N} C_{\beta,\ell,p,k}(x,\xi) z^p h^\beta \pf(Q(\t M(x, 0)^{-1}\xi)-i0)^{-k-1-\ell-p}  +R_{k,N}(z,x,h;\xi),
\eea
$$
where $C_{\beta,\ell,p,k}$ depends smoothly on $x$ and is a universal polynomial in $\xi$ of degree $2\ell$, $\beta$ is a multi-index, the coefficients  
of $C_{\beta,\ell,p,k}$ are combinatorially defined from the above expansions depending on derivatives of $M(x,h)$ in $h$ at $h=0$.
It is a crucial fact that the remainder $R_{k,N}(x,h;\xi)$ is a distribution weakly homogeneous in $\xi$ of degree $\geqslant k$,  and
vanishes at order at least $N-k$ in $h$.  
The important fact is that  
$R_{k,N}(z,x,h;\xi)$  is an element in $C^\infty\left( \Omega\right)\otimes \pazocal{S}^\prime\left(\mathbb{R}^{n}\right)$ and 
$\big( \lambda^{-N-1} R_{k,N}(z,x,\lambda h;\frac{\xi}{\lambda})\big)_{\lambda\in \opencl{0,1}} $ is 
bounded in $C^\infty\left( \Omega\right)\otimes \pazocal{S}^\prime\left(\mathbb{R}^{n}\right)$.

Finally, we get
$$
\bea
H_N(z)&=\sum_{2(k+1)+2p+\vert \beta\vert\leqslant N} \frac{k!\left(\chi u_k\right)(x,h)h^\beta\module{M(x,h)}^{-1}{(-1)^p}z^p}{(2\pi)^n\beta !} \begin{pmatrix}
-k-1\\
p,-k-1-p
\end{pmatrix}  \fantom\phantom{=========} \times \int_{\mathbb{R}^n}e^{i\left\langle\xi,h\right\rangle}
\partial^\beta_h \pf(Q(\t M(x, h)^{-1}\xi)-i0)^{-k-1-p}|_{(x,0)}d^n\xi
\fantom +\int_{\mathbb{R}^n}e^{i\left\langle\xi,h\right\rangle}R_{1,N}(z,x,h;\xi)d^n\xi  +R_{2,N}(z,x,h),
\eea
$$
where $R_{2,N}(z,x,h)\in \pazocal{C}^{s}\left(\Omega\right)$ is a function of 
H\"older regularity $s$ which can be made arbitrarily large by choosing $N$ large enough, the term 
$R_{1,N}(z,x,h;\xi)$ is an element in $C^\infty\left( \Omega\right)\otimes \pazocal{S}^\prime\left(\mathbb{R}^{n}\right)$,
such that the family
$\big( \lambda^{-N-1} R_{1,N}(z,x,\lambda h;\frac{\xi}{\lambda})\big)_{\lambda\in \opencl{0,1}} $ is 
bounded in $C^\infty\left( \Omega\right)\otimes \pazocal{S}^\prime\left(\mathbb{R}^{n}\right)$.
It follows that $\Pi_0\left(R_{1,N}  \right)=\euler\Pi_0\left(R_{2,N} \right)=0$ if $N$ is chosen large enough.
It is clear from the construction   that the terms
$\int_{\mathbb{R}^n}e^{i\left\langle\xi,h\right\rangle}
\partial^\beta_h \pf(Q(\t M(x, h)^{-1}\xi)-i0)^{-k-1-p}|_{(x,0)}d^n\xi$ are quasihomogeneous
and multiplying by smooth functions preserves the tame log-polyhomogeneity. This finishes the proof.
\end{refproof}

\subsection{Residue computation and conclusions} \label{ss:rcc} Now that we know $H_N(z)$ is tame log-polyhomogeneous,  our next objective is to extract the term
$\euler \Pi_0(H_N(z))$ and express it in terms of the Hadamard coefficients $(u_k)_{k=0}^\infty$.


 We first prove a key lemma related to the extraction of the dynamical residues which shows  that the residue of many terms vanishes.  { Recall that the notion of finite part $\pf$ was introduced in Definition~\ref{d:finitepart}.} 

\begin{lemm}\label{l:vanishlemma2}
Let $\euler=h^i\partial_{h^i}$, $\varphi\in C^\infty(\Omega)$, $\beta=(\beta_1,\dots,\beta_\ell)\in \nn^\ell$, $k\in \mathbb{N}$ and let $P$ be a homogeneous polynomial on $\mathbb{R}^n$ of { even
degree}.
Then the residue
$$
\euler \Pi_0\left(h^\beta \varphi\int_{\mathbb{R}^n}P(\xi) e^{i\left\langle\xi,h\right\rangle}\pf(Q(\xi)-i0)^{-k}d^n\xi \right)
$$
vanishes if $-2k+\deg(P)\neq -n$ or  $\vert \beta\vert>0$. On the other hand, in the special case $-2k=-n$, 
$$
\euler \Pi_0\left(\varphi \int_{\mathbb{R}^n}  e^{i\left\langle\xi,h\right\rangle}\pf(Q(\xi)-i0)^{-k} d^n\xi\right)=\varphi(x,0)\int_{\mathbb{S}^{n-1}}
 (Q(\xi)-i0)^{-k} \iota_\euler d^n\xi.
$$
\end{lemm}
\begin{remark}\label{rem:evalat0}
Note that the projector $\Pi_0$ has the effect of evaluating the test function
$\varphi$ at $h=0$.
\end{remark}

\begin{proof}
The important fact is that $P(\xi)\pf(Q(\xi)-i0)^{-k}$ is a quasihomogeneous 
distribution in the $\xi$ variable.
By Taylor expansion of $\varphi$ in the $h$ variable, we get for any $N$:
$$ \bea
& h^\beta\varphi \int_{\mathbb{R}^n} P(\xi) e^{i\left\langle\xi,h\right\rangle}\pf(Q(\xi)-i0)^{-k} d^n\xi
\\ &=\sum_{\vert\beta_2\vert\leqslant N} \frac{h^{\beta+\beta_2}}{\beta_2!} \partial_h^{\beta_2}\varphi(x,0) \int_{\mathbb{R}^n}  e^{i\left\langle\xi,h\right\rangle}P(\xi)\pf(Q(\xi)-i0)^{-k} d^n\xi
\fantom +\sum_{\vert \beta_2\vert=N+1} h^{\beta+\beta_2} R_{\beta_2}(x,h)  \int_{\mathbb{R}^n}  e^{i\left\langle\xi,h\right\rangle}
P(\xi)\pf(Q(\xi)-i0)^{-k} d^n\xi.
\eea
$$
By scaling, if $\vert \beta_2\vert=N+1$ then
$$\left\langle { e^{-t\euler}}\bigg( h^{\beta+\beta_2} R_{\beta_2}(x,h)  \int_{\mathbb{R}^n} e^{i\left\langle\xi,h\right\rangle}\pf(Q(\xi)-i0)^{-k} d^n\xi\bigg), \psi \right\rangle=\pazocal{O}(e^{-t((N+1)-2k+n -\varepsilon)}) $$ 
for all $\varepsilon>0$ which accounts for the corrective behaviours of polynomials in $t$ produced by the Jordan blocks. Then choosing $N$ large enough, we can take the Laplace transform 
$$\int_0^\infty e^{-tz}  \left\langle { e^{-t\euler}}\bigg(  h^{\beta+\beta_2} R_{\beta_2}(x,h)  \int_{\mathbb{R}^n} e^{i\left\langle\xi,h\right\rangle}\pf(Q(\xi)-i0)^{-k} d^n\xi\bigg), \psi \right\rangle dt $$
holomorphic for $z$ near $0$. Therefore since the projector $\Pi_0$ is defined by contour integration using Cauchy's formula, we get
that 
\begin{eqnarray*} \Pi_0\left(  h^{\beta+\beta_2} R_{\beta_2}(x,h)   \int_{\mathbb{R}^n} e^{i\left\langle\xi,h\right\rangle}\pf(Q(\xi)-i0)^{-k} d^n\xi\right)=0.
\end{eqnarray*}

The provisional conclusion is that we need to inspect the expression
$$
\bea 
&\Pi_0\left( h^\beta  \int_{\mathbb{R}^n}  e^{i\left\langle\xi,h\right\rangle}P(\xi) \pf(Q(\xi)-i0)^{-k} d^n\xi\right) \\ &=
\Pi_0\left( i^{-\vert\beta\vert} \int_{\mathbb{R}^n}  e^{i\left\langle\xi,h\right\rangle} \partial_\xi^\beta P(\xi) \pf(Q(\xi)-i0)^{-k} d^n\xi\right).
\eea
$$
If $-\vert\beta \vert-2k+\deg(P)\neq -n$, the current $ \partial_\xi^\beta P(\xi) \pf(Q(\xi)-i0)^{-k} d^n\xi$ is quasihomogeneous of degree $-\vert\beta \vert-2k+n+\deg(P) $ hence its inverse Fourier transform is also quasihomogeneous of degree $p\neq 0$ and therefore its image under the projector 
$\Pi_0$ vanishes. 

If $\vert \beta\vert+2k=n, \vert\beta\vert>0$, then {Corollary}~\ref{l:vanishing1} together with Lemma~\ref{l:scalanomal} imply that
\begin{eqnarray*}
\euler\Pi_0\left(i^{-\vert\beta\vert} \int_{\mathbb{R}^n}  e^{i\left\langle\xi,h\right\rangle} \partial_\xi^\beta \pf(Q(\xi)-i0)^{-k} d^n\xi \right)
= \int_{\vert \xi\vert=1}  \partial_\xi^\beta \pf(Q(\xi)-i0)^{-k} \iota_\Feuler d^n\xi =0.
\end{eqnarray*}
Finally, when $2k=n$ and $\vert\beta\vert=0$ Lemma~\ref{l:scalanomal} implies that
the residue equals
$$\euler\Pi_0\left( \int_{\mathbb{R}^n}  e^{i\left\langle\xi,h\right\rangle} \pf(Q(\xi)-i0)^{-k} d^n\xi \right)=\int_{\mathbb{S}^{n-1}}
(Q(\xi)-i0)^{-k} \iota_\Feuler d^n\xi$$
as claimed. 
\end{proof}

Now, Lemma \ref{l:vanishlemma2} applied to $H_N(z)$ gives first
$$ \bea
&\euler\Pi_0 \big( H_N(z)\big)\\
&=\sum_{2k+2+2p+\vert \beta\vert\leqslant N}\begin{pmatrix}
-k-1\\
p,-k-1-p
\end{pmatrix} k! {(-1)^p}  z^p
\fantom \times  \euler\Pi_0\left( \frac{\left(\chi u_k\right)(x,h)h^\beta\module{M(x,h)}^{-1}}{(2\pi)^n\beta !}  \int_{\mathbb{R}^n}e^{i\left\langle\xi,h\right\rangle}
\partial^\beta_h \pf(Q(\t M(x, h)^{-1}\xi)-i0)^{-k-1-p}|_{(x,0)}d^n\xi\right)
\\
&=
\sum_{ 2k+2+2p=n }
 \euler\Pi_0\bigg( \frac{k!\left(\chi u_k\right)(x,h)\module{M(x,h)}^{-1}{(-1)^p}  z^p}{(2\pi)^n} \begin{pmatrix}
-k-1\\
p,-k-1-p
\end{pmatrix}  \bigg. \\  &\bigg.\phantom{==========================}\times \int_{\mathbb{R}^n}e^{i\left\langle\xi,h\right\rangle}
 \pf(Q(\xi)-i0)^{-k-1-p}|_{(x,0)}d^n\xi\bigg)
\eea
$$
where we used the fact that the projector $\Pi_0$ evaluates $\left(\chi u_k\right)(x,h)\module{M(x,h)}^{-1}$ at $h=0$ by Remark~\ref{rem:evalat0} and that $M(x,0)=\id, \chi(x,0)=1$, and then we obtain the shorter expression~\footnote{{We used here the identity 
$\begin{pmatrix}
-k-1\\
p,-k-1-p
\end{pmatrix} k!= \frac{(-k-1)\dots (-k-p) }{p!}k!=(-1)^p\frac{k+p!}{p!} $ .}} 
\beq\label{eq:shorter} 
\bea 
\euler\Pi_0 \big(H_N(z)\big)&=\sum_{ 2k+2p+2=n } \begin{pmatrix}
-k-1\\
p,-k-1-p
\end{pmatrix}  
  \frac{k!u_k(x,0){(-1)^p}z^p}{(2\pi)^n}\int_{\mathbb{S}^{n-1}}
(Q(\xi)-i0)^{-\frac{n}{2}} \iota_\Feuler d^n\xi\\
&=\sum_{ 2k+2p+2=n } 
  \frac{(k+p)!u_k(x,0)z^p}{p!(2\pi)^n}\int_{\mathbb{S}^{n-1}}
(Q(\xi)-i0)^{-\frac{n}{2}} \iota_\Feuler d^n\xi.
\eea 
\eeq
Finally, to get a more direct expression for  $\euler\Pi_0 \big( H_N(z)\big)$ we need to compute the integral on the r.h.s.

\begin{lemm}[Evaluation of the residue by Stokes theorem]  \label{lem:Stokes}
We have the identity:
\begin{equation}\label{restokes}
\int_{\mathbb{S}^{n-1}}
( -\xi_1^2+\xi_2^2+\dots+\xi_n^2 -i0)^{-\frac{n}{2}} \iota_\Feuler d^n\xi=\frac{2i\pi^{\frac{n}{2}}}{\Gamma(\frac{n}{2})}.
\end{equation}
\end{lemm}
\begin{proof}
The proof follows by a Wick rotation argument as in  \cite[\S8.3]{Dang2020}.
We  complexify the whole setting and 
define the holomorphic $(n-1,0)$-form:
\begin{eqnarray*}
\omega=\left( z_1^2+\dots+z_n^2\right)^{-\frac{n}{2}}\iota_{\sum_{i=1}^n z_i\partial_{z_i}} dz_1\wedge \dots\wedge dz_n \in \Omega^{n-1,0}\left(  U \right),
\end{eqnarray*}
where $U$ is the Zariski open subset $\{z\in \mathbb{C}^n \st Q(z)\neq 0 \}$. By the Lie--Cartan formula
$$
\pazocal{L}_{\sum_{i=1}^n z_i\partial_{z_i}} =d \iota_{\sum_{i=1}^n z_i\partial_{z_i}}+\iota_{\sum_{i=1}^n z_i\partial_{z_i}}d,
$$
and
$$ d \left( z_1^2+\dots+z_n^2\right)^{-\frac{n}{2}} dz_1\wedge \dots\wedge dz_n=0 \in\Omega^{n,1}(U),
$$
hence
$$
\bea
&\pazocal{L}_{\sum_{i=1}^n z_i\partial_{z_i}} \left( z_1^2+\dots+z_n^2\right)^{-\frac{n}{2}} dz_1\wedge \dots\wedge dz_n \\ &= d\left( z_1^2+\dots+z_n^2\right)^{-\frac{n}{2}}\iota_{\sum_{i=1}^n z_i\partial_{z_i}} dz_1\wedge \dots\wedge dz_n=0,
\eea
$$ 
so the differential form $\omega$ is closed in $\Omega^{n-1,0}(U)$.
For every $\theta\in \clopen{0,-\frac{\pi}{2}}$,  
we define the $n$-chain
$$ E_\theta=\{ (e^{iu}z_1,z_2,\dots,z_n) \st (z_1,\dots,z_n)\in \mathbb{S}^{n-1}\subset\mathbb{R}^n, \ u\in [\theta,0] \}  $$
which is contained in $\mathbb{S}^{2n-1}$.
We denote by $\partial$ the boundary operator acting on de Rham currents, under some choice of orientation on $E_\theta$, we have the equation 
$$
\partial E_\theta=[P_\theta  ] -[P_0],
$$
where $[P_\theta]$ denotes the current of integration on the $(n-1)$-chain 
$$
P_\theta=\{(e^{i\theta}z_1,z_2,\dots,z_n) \st (z_1,\dots,z_n)\in \mathbb{S}^{n-1}\subset \mathbb{R}^n\}.
$$
By Stokes theorem, 
\begin{eqnarray*}
0= \int_{E_\theta} d\omega =\int_{ \partial E_\theta } \omega=\int_{P_\theta} \omega-\int_{P_0}\omega
\end{eqnarray*}
where the integration by parts is well-defined
since for $\theta\in \clopen{0,-\frac{\pi}{2}}$, the zero locus of 
$ \sum_{i=1}^nz_i^2$ never meets $P_\theta$ so we are integrating well-defined smooth forms~\footnote{Indeed, if $\theta\in \open{0,-\frac{\pi}{2}}$ and $e^{2i\theta}z_1^2+z_2^2+\dots+z_n^2=0$ then $\sin(2\theta)z_1^2=0$, hence $z_1=0$ and $\sum_{i=1}^n z_i^2=0$, which contradicts the fact that $(z_1,\dots,z_n)\in \mathbb{S}^{n-1}$.}.
 
We define the linear automorphism
$ T_\theta: (z_1,\dots,z_n)\mapsto (e^{i\theta}z_1,\dots,z_n) $ and
note that 
$$
\bea
\int_{P_\theta} \omega&=\int_{P_0} T_\theta^*\omega=
e^{i\theta}\int_{\mathbb{S}^{n-1}}
( e^{i2\theta}\xi_1^2+\xi_2^2+\dots+\xi_n^2)^{-\frac{n}{2}} \iota_\Feuler d^n\xi \\ &=\int_{\mathbb{S}^{n-1}}
(\xi_1^2+\xi_2^2+\dots+\xi_n^2)^{-\frac{n}{2}} \iota_\Feuler d^n\xi=\Vol(\mathbb{S}^{n-1}).
\eea
$$

By \cite[Lem.~D.1]{Dang2020}, 
$$( e^{i2\theta}\xi_1^2+\xi_2^2+\dots+\xi_n^2)^{-\frac{n}{2}}\rightarrow (Q(\xi)-i0)^{-\frac{n}{2}}  \mbox{ in } \pazocal{D}^\prime_\Gamma(\mathbb{R}^n\setminus \{0\})$$ as $\theta\rightarrow -\frac{\pi}{2}$,   where
$\Gamma=\{ (\xi;\tau dQ(\xi)) \st Q(\xi)=0,\tau<0 \}$  is the half-conormal of the cone $\{Q=0\}$. Since $\Gamma\cap N^*\mathbb{S}^{n-1}=\emptyset$, 
in the limit we obtain $$\lim_{\theta\rightarrow -\frac{\pi}{2}^+} \int_{\mathbb{S}^{n-1}}
( e^{i2\theta}\xi_1^2+\xi_2^2+\dots+\xi_n^2)^{-\frac{n}{2}} \iota_\Feuler d^n\xi=  \left\langle [\mathbb{S}^{n-1}], (Q(\xi)-i0)^{-\frac{n}{2}} \iota_\Feuler d^n\xi \right\rangle $$ where the distribution pairing is well-defined
by transversality of wavefront sets.
From this we conclude  \eqref{restokes}.
\end{proof}

Combining  \eqref{eq:shorter}  with  Lemma \ref{lem:Stokes} gives us  
\beq\label{eq:shorter2}
\bea 
\euler\Pi_0 \big( H_N(z)\big)&=\sum_{ 2k+2p+2=n } 
  \frac{(k+p)!u_k(x,0)z^p}{p!(2\pi)^n}\int_{\mathbb{S}^{n-1}}
(Q(\xi)-i0)^{-\frac{n}{2}} \iota_\Feuler d^n\xi\\
&=\frac{2i\pi^{\frac{n}{2}}}{\Gamma(\frac{n}{2})}\sum_{ 2k+2p+2=n } 
  \frac{(k+p)!u_k(x,0)z^p}{p!(2\pi)^n},
\eea 
\eeq
from which we obtain { the following result.

\begin{prop}
Let $H_N(z)$ be the Hadamard parametrix of order $N$. Then for any Euler vector field $\euler$, the
dynamical residue satisfies
$$
\resdyn \big(H_N(z)\big)=i\sum_{p=0}^{\frac{n}{2}-1}
\frac{z^pu_{\frac{n}{2}-p-1}(x,0)}{p! 2^{n-1}\pi^{\frac{n}{2}}}.
$$
In particular,  $\resdyn \big(H_N(z)\big)$ is independent on the choice of Euler vector field $X$.
\end{prop}
}

\section{Residues of local and spectral Lorentzian zeta functions} \label{section5}

\subsection{Hadamard parametrix for complex powers} \label{ss:hc} As previously, we consider the wave operator $P=\square_g$ on a  { time-oriented}  Lorentzian manifold $(M,g)$ of even dimension $n$. Just as the Hadamard parametrix $H_N(z)$ is designed to approximate Feynman inverses of $P-z$ near the diagonal, we can construct a more general parametrix  $H_N^{(\cv)}(z)$ for $\cv\in\cc$ which is meant as an approximation (at least formally) of complex powers $(P-z)^{-\cv}$.  

To motivate the definition of $H_N^{(\cv)}(z)$, let us recall that if $A$ is a self-adjoint operator in a Hilbert space  then for all $z=\mu+i\varepsilon$ with $\mu\in\rr$ and $\varepsilon>0$,
$$
(A-z)^{-\cv}=\frac{1}{2\pi i}\int_{\gamma_\varepsilon} (\lambda-i\varepsilon)^{-\cv} (A-\mu-\lambda)^{-1} d\lambda
$$
in the strong operator topology   (see e.g.~ \cite[App.~B]{Dang2020}).  The contour of integration  $\gamma_\varepsilon$ is represented in Figure \ref{fig:contour} and can be written as $\gamma_\varepsilon= \tilde\gamma_\varepsilon+i\varepsilon$, where
      \beq
      \tilde\gamma_{\varepsilon} = e^{i(\pi-\theta)}\opencl{-\infty,\textstyle\frac{\varepsilon}{2}}\cup \{\textstyle\frac{\varepsilon}{2} e^{i\omega}\, | \, \pi-\theta<\omega<\theta\}\cup  e^{i\theta}\clopen{\textstyle\frac{\varepsilon}{2},+\infty}
      \eeq
       goes from $\Re \lambda\ll 0$ to  $\Re \lambda\gg 0$ in the upper half-plane (for some fixed  $\theta\in\open{0,\pid}$).
 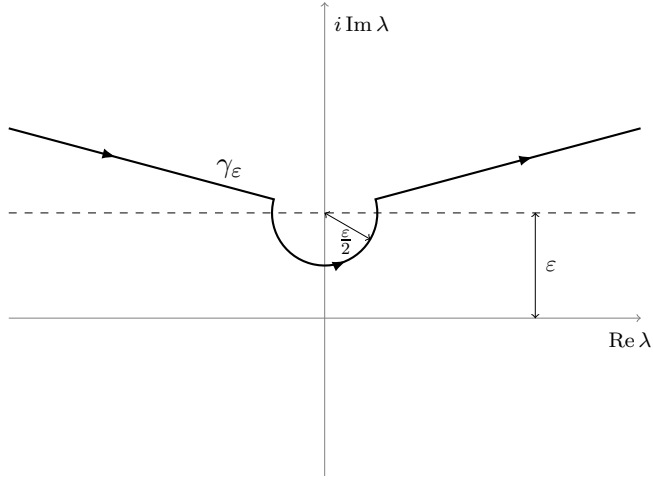
\begin{figure}
 \begin{tikzpicture}[scale=1.4]
 \def\bigradius{3}
 \def\incangle{15}
 
 \def\littleradius{0.5}
 
 \draw [help lines,->] (-1.0*\bigradius, 0) -- (1.0*\bigradius,0);
 \draw [help lines,->] (0, -0.5*\bigradius) -- (0, 1.0*\bigradius);
 
 \begin{scope}[shift={(0,2*\littleradius)}]
  \node at (-0.9,0.42) {$\gamma_{\varepsilon}$};
\path[draw,line width=0.8pt,decoration={ markings,
      mark=at position 0.15 with {\arrow{latex}}, 
      mark=at position 0.53 with {\arrow{latex}},
      mark=at position 0.85 with {\arrow{latex}}},postaction=decorate] (-\bigradius,{\bigradius*tan(\incangle)})   -- (-\incangle:-\littleradius) arc (180-\incangle:360+\incangle:\littleradius)   -- (\bigradius,{\bigradius*tan(\incangle)});
      \path[draw,line width=0.2pt,postaction=decorate,<->] (0,0)   -- (330:\littleradius);
\end{scope}

\path[draw,line width=0.2pt,postaction=decorate,dashed] (-\bigradius,2*\littleradius)   -- (\bigradius,2*\littleradius);

\path[draw,line width=0.2pt,postaction=decorate,<->] (2,0)   -- (2,2*\littleradius);
\node at (2.15,\littleradius){\scalebox{0.8}{$\varepsilon$}};
\node at (0.2,1.45*\littleradius){\scalebox{0.8}{$\frac{\varepsilon}{2}$}};

 \node at (2.9,-0.2){$\scriptstyle \Re \lambda$};
 \node at (0.35,2.8) {$\scriptstyle i \Im \lambda$};
 
 \end{tikzpicture}
 \caption{\label{fig:contour}The contour $\gamma_\varepsilon$ used to write $(A-i\varepsilon)^{-\cv}$ as an integral of the resolvent $(A-\lambda)^{-1}$ for $A$ self-adjoint.} 
 \end{figure}
This suggests immediately to set 
$$
\bea
H_N^{(\cv)}(z;.)&\defeq \frac{1}{2\pi i}\int_{\gamma_\varepsilon} (\lambda-i\varepsilon)^{-\cv} H_N(\mu+\lambda,.) d\lambda \\
& =\sum_{\varm=0}^N \chi u_\varm\frac{1}{2\pi i}\int_{\gamma_{\varepsilon}}(\lambda-i\varepsilon)^{-\cv}  \mathbf{F}_\varm(\mu+\lambda,.) d\lambda
\eea
$$
provided that the r.h.s.~makes sense. For $\Re \cv >0$ the integral converges by  the estimate in \cite[Lem.~6.1]{Dang2020}. More generally, we can evaluate the integral thanks to the identity
$$
\frac{1}{2\pi i}\int_{\gamma_\varepsilon}(\lambda-i\varepsilon)^{-\cv} \mathbf{F}_\varm(\mu+\lambda,.) d\lambda =\frac{(-1)^\varm\Gamma(-\cv+1)}{\Gamma(-\cv-\varm+1)\Gamma(\cv+\varm)}   \mathbf{F}_{\varm+\cv-1}(\mu+i\varepsilon,.)
$$
shown in  \cite[\S7.1]{Dang2020}, and use it to analytically continue $H_N^{(\cv)}(z)=H_N^{(\cv)}(\mu+i\varepsilon)$. This gives 
$$
H_N^{(\cv)}(z,.) = 
 \sum_{\varm=0}^N  u_\varm(.)\frac{(-1)^\varm\Gamma(-\cv+1)}{\Gamma(-\cv-\varm+1)\Gamma(\cv+\varm)} \mathbf{F}_{\varm+\cv-1}(z,.)
 $$
as a distribution in a neighborhood of $\Delta\subset M\times M$.




From now on the analysis in \secs{ss:polhom}{ss:rcc} can be applied with merely minor changes. For the sake of brevity we write `$\sim$' to denote identities which hold true modulo remainders as those discussed in \secs{ss:polhom}{ss:rcc}, which do not contribute to residues. In particular we can write
$$
H_N^{(\cv)}(z) \sim \sum_{k=0}^N u_k \frac{\cv\dots(\cv+k-1)}{(2\pi)^n} \int_{\mathbb{R}^n} e^{i\left\langle\xi,h \right\rangle}
(Q(\t M^{-1}(x,h)\xi)-i0)^{-k-\cv} \module{M(x,h)}^{-1} d^n\xi.
$$
Expanding in $z$ yields
$$ 
\bea
H_N^{(\cv)}(z)&\sim \sum_{k=0}^N\sum_{p=0}^\infty u_k  (-1)^pz^p  \begin{pmatrix}
-k-\cv
\\
p
\end{pmatrix}   \frac{\cv\dots(\cv+k-1)}{(2\pi)^n} \fantom \times  \int_{\mathbb{R}^n} e^{i\left\langle\xi,h \right\rangle}
(Q(\t M^{-1}(x,h)\xi)-i0)^{-k-\cv-p} \module{M(x,h)}^{-1} d^n\xi
\\ & \sim  \sum_{  k=0}^\infty \sum_{p=0}^\infty z^p  u_k \frac{\cv\dots(\cv+k+p-1)}{p!(2\pi)^n} \fantom\times\int_{\mathbb{R}^n} e^{i\left\langle\xi,h \right\rangle} 
(Q(\t M^{-1}(x,h)\xi)-i0)^{-k-\cv-p} \module{M(x,h)}^{-1} d^n\xi.
\eea
$$
We  take the dynamical residue and in view of   Lemma \ref{l:vanishlemma2},   only the    
terms with $\cv+k+p=\frac{n}{2}$  survive. We find that for $\cv=0,\dots,\n2$, the dynamical residue $\resdyn\big(H_N^{(\cv)}(z)\big)$ equals
$$ \bea
&\sum_{k+p+\cv=\frac{n}{2}}
z^p   \frac{\cv\dots(\frac{n}{2}-1)}{p!(2\pi)^n}  \resdyn\left( u_k\int_{\mathbb{R}^n} e^{i\left\langle\xi,h \right\rangle}
(Q(\t M^{-1}(x,h)\xi)-i0)^{-k-\cv-p} \module{M(x,h)}^{-1} d^n\xi\right)
\\
&=\sum_{p=0}^{\frac{n}{2}-\cv}
z^p   \frac{\cv\dots(\frac{n}{2}-1)}{p!(2\pi)^n}  \resdyn\left( u_{\frac{n}{2}-p-\cv}\int_{\mathbb{R}^n} e^{i\left\langle\xi,h \right\rangle}
(Q(\t M^{-1}(x,h)\xi)-i0)^{-\frac{n}{2}} \module{M(x,h)}^{-1} d^n\xi\right)\\
&=\frac{2i\pi^{\frac{n}{2}}}{\Gamma(\frac{n}{2})}\sum_{p=0}^{\frac{n}{2}-\cv}
\frac{z^p}{p!}   \frac{\cv\dots(\frac{n}{2}-1)}{(2\pi)^n}u_{\frac{n}{2}-p-\cv}(x,x).
\eea
$$
In consequence, we obtain 
\beq\label{eq:resdynHa}
\resdyn\big(H_N^{(\cv)}(z)\big)=
i\sum_{p=0}^{\frac{n}{2}-\cv}
\frac{z^pu_{\frac{n}{2}-p-\cv}(x,x)}{p!(\cv-1)!2^{n-1}\pi^{\frac{n}{2}}} . 
\eeq
On the other hand, from \cite[\S8.3.1]{Dang2020} we know that for $N$ sufficiently large
\beq\label{eq:resdynHa2}
\res_{\alpha'=\alpha} \big( \iota^*_\Delta H_N^{(\alpha')}(z)\big)=
i\sum_{p=0}^{\frac{n}{2}-\cv}
\frac{z^p u_{\frac{n}{2}-p-\cv}(x,x)}{p!(\cv-1)!2^{n}\pi^{\frac{n}{2}}}  
\eeq
where the residue is understood in the sense of complex analysis.
We summarize this as a proposition.

\begin{proposition} For any Euler vector field $\euler$, there exists  an $\euler$-stable neighborhood $\pazocal{U}$ of $\Delta\subset M\times M$ such that $H_N^{(\cv)}(z)\in \pazocal{D}^\prime(\pazocal{U})$ 
is tame log-polyhomogeneous w.r.t.~$\euler$. The dynamical residue $\resdyn\big(H_N^{(\cv)}(z)\big)$ is {independent of} $\euler$ and satisfies
\beq\label{eq:resres}
\resdyn\big(H_N^{(\cv)}(z)\big)={2}\res_{\alpha'=\alpha} \big( \iota^*_\Delta H_N^{(\alpha')}(z)\big)
\eeq
where the residue on the r.h.s.~is understood in the sense of complex analysis.
For $\alpha=0,\dots,\n2$ it has the explicit expression \eqref{eq:resdynHa}.
\end{proposition}

\begin{remark} The parametrix $H_N^{(\cv)}(i\varepsilon)$ is interpreted as a local (and for the moment purely formal) approximation of $(\square_g-i \varepsilon)^{-\cv}$, and  similarly if we define 
$$
\zeta_{g,\varepsilon}^{\loc}(\cv)=  \iota^*_\Delta H_{N(\alpha)}^{(\alpha)}(i \varepsilon)
$$
where $N(\alpha)$ is taken sufficiently large, $\zeta_{g,\varepsilon}^{\loc}(\cv)$ can be seen as a local  approximation of the Lorentzian spectral zeta function density $\zeta_{g,\varepsilon}(\cv)$   studied in the next section.
\end{remark}



\subsection{From local to spectral zeta functions}\label{ss:lg} 

Let us now analyze what happens in situations when $P=\square_g$ (or  strictly speaking, $P-i \varepsilon$) has a well-defined spectral zeta function density $\zeta_{g,\varepsilon}(\cv)$ in the following sense.

\bed Suppose $P$ is a self-adjoint extension of $\square_g$ acting on $C_\c^\infty(M)$. Then, the spectral zeta function density of $P-i\varepsilon$ is the meromorphic continuation of
$$
\cv\mapsto \zeta_{g,\varepsilon}(\cv)= \iota^*_\Delta \big((P-i \varepsilon)^{-\cv} \big),
$$
defined initially for  $\Re \cv$ sufficiently large, where $\iota^*_\Delta$ is the pull-back of the Schwartz kernel to the diagonal $\Delta\subset M\times M$.
\eed

It is a priori not clear whether the definition is useful at all because even if a self-adjoint extension $P$ exists, it is by far not evident { whether the Schwartz kernel of $(P-i \varepsilon)^{-\cv}$ has a well-defined restriction} to the diagonal for large $\Re \cv$, not to mention the analyticity aspects.

We can however formulate a natural sufficient condition in the present context. We start by stating a definition of the uniform wavefront set (which is equivalent to \cite[Def.~3.2]{Dang2020}). Below, $\zero$ is the zero section of $T^*M$ and $\bra z\ket=(1+\module{z}^2)^\12$.

\begin{definition}\label{defrrr}
The \emph{uniform operator wavefront set of order $s\in\rr$ and weight $\bra z\ket^{-\12}$} of $(P-z)^{-1}$ is the set 
\beq\label{eq:wfs}
\wfl{12}((P-z)^{-1})\subset (T^*M\setminus\zero)\times (T^*M\setminus\zero)
\eeq
defined as follows:  $((x_1;\xi_1),(x_2;\xi_2))$ does \emph{not} belong to \eqref{eq:wfs} if and only if for all $\varepsilon>0$ and all  properly supported $B_i\in \Psi^{0}(M)$ elliptic at $(x_i,\xi_i)$ and all $r\in \rr$, 
$$
\bra z\ket^{\12}B_1(P-z)^{-1} B_2^* \mbox{ is bounded in } B(H^{r}_\c(M), H_\loc^{r+s}(M)) \mbox{ along } z\in \gamma_\varepsilon.
$$
\end{definition}

The key property which we require of $\square_g$ is that it has a self-adjoint extension $P$, and that self-adjoint has \emph{Feynman wavefront set} in the sense of the uniform operator wavefront set. More precisely, we formalize this as follows.

\begin{definition}\label{deff} Suppose $P$ is a self-adjoint extension of $\square_g$ acting on $C_\c^\infty(M)\subset L^2(M)$. We say that $\square_g$ has \emph{Feynman resolvent} if for any $s\in\rr$,  the family  $\{(P-z)^{-1}\}_{z\in \gamma_\varepsilon}$ satisfies
\beq\label{feynwf2}
\bea
\wfl{12}  \big( ( P-z)^{-1} \big)\subset \{  ((x_1;\xi_1),(x_2;\xi_2))\, | \, (x_1;\xi_1) { \succeq} (x_2;\xi_2) \mbox{ or } x_1=x_2 \}.
\eea
\eeq
Above, $(x_1;\xi_1){ \succeq} (x_2;\xi_2)$ means that $(x_1;\xi_1)$ lies in the characteristic set of $P$  and $(x_1;\xi_1)$ can be joined from $(x_2;\xi_2)$ by a forward{\footnote{{ We remark that the opposite convention for the Feynman wavefront set is often used in the literature on Quantum Field Theory on curved spacetimes. Note also that the notion of {forward} vs.~backward bicharacteristic depends on the sign convention for $P$ (or rather its principal symbol).}}} bicharacteristic.
\end{definition}

This type of precise information on the microlocal structure of  $(P-z)^{-1}$ allows one to solve away the singular error term $r_N(z)$ which appears in \eqref{eq:PzHN} when computing $\left(P-z\right)  H_N(z)$. In consequence, the Hadamard parametrix approximates $(P-z)^{-1}$ in the following uniform sense.

\begin{proposition}[{\cite[Prop.~6.3]{Dang2020}}]\label{prop:fh} If $\square_g$ has Feynman resolvent then for every $s,\ell\in \mathbb{R}_{\geqslant 0}$,
there exists $N$ such that
\begin{eqnarray}\label{eq:toinsert}
(P-z)^{-1}= H_N(z) +E_{N,1}(z)+E_{N,2}(z)
\end{eqnarray}
where for $z$ along $\gamma_\varepsilon$, $\bra  z \ket^k \tilde\chi E_{N,1}(z)$ is bounded in $\cD'(M\times M)$ for some  $\tilde\chi\in \cf(M\times M)$ supported near $\Delta$ and all $k\in  \mathbb{Z}_{\geqslant 0}$, 
and $\bra  z \ket^\ell  E_{N,2}(z)$ is bounded in $B(H^{r}_\c(M), H_\loc^{r+s}(M))$ for all $r\in \rr$. 
\end{proposition}

Then, by integrating  $(z-i\varepsilon)^{-\cv}$ times both sides of \eqref{eq:toinsert} along the contour $\gamma_\varepsilon$ we obtain for all $z$, 
\beq
(P-z)^{-\cv}= H_N^{(\cv)}(z) +R_N^{(\cv)}(z),
\eeq
where for each $s\in\rr$ and $p\in \nn$ there exists $N\in \nn$ such that $R_N^{(\cv)}(z)$ is holomorphic in $\{\Re \cv >-p\}$ with values in $C^s_{\rm loc}(\pazocal{U})$. Thus, the error term does not contribute to  neither analytical nor dynamical residues. By  combining { all the above information} with Proposition \ref{prop:fh} we obtain the following final result.

\begin{thm}\label{thm:dynres1}  Let $(M,g)$ be a  { time-oriented}  Lorentzian manifold of even dimension $n$, and suppose $\square_g$ has Feynman resolvent $(P-z)^{-1}$. Then for any Euler vector field $\euler$ there exists  an $\euler$-stable neighborhood $\pazocal{U}$ of $\Delta\subset M\times M$ such that  for all $\cv\in\cc$ and $\Im z > 0$ the Schwartz kernel $K_\cv\in \pazocal{D}^\prime(\pazocal{U})$ of $(P-z)^{-\cv}$
is tame log-polyhomogeneous w.r.t.~scaling with $\euler$. 
The dynamical residue of $(P-z)^{-\cv}$ is {independent of} $\euler$ and equals
\beq\label{eq:explicit}
\resdyn\big(\left(P-z\right)^{-\cv} \big)= i\sum_{p=0}^{\frac{n}{2}-\cv}
\frac{z^pu_{\frac{n}{2}-p-\cv}(x,x)}{p!(\cv-1)!2^{n-1}\pi^{\frac{n}{2}}}
\eeq
if $\alpha = 1,\dots, \n2$, and zero otherwise, where $(u_j(x,x))_j$ are the Hadamard coefficients. Furthermore, for $k=1,\dots,\frac{n}{2}$ and $\varepsilon>0$, the dynamical residue satisfies
\beq\label{eq:mainn}
\resdyn \left(P-i \varepsilon \right)^{-k} = 2 \res_{\cv =k}\zeta_{g,\varepsilon}(\cv),
\eeq
where $\zeta_{g,\varepsilon}(\cv)$ is the spectral zeta function density of $P-i \varepsilon$. 
\end{thm}

In particular, using the fact that $u_1(x,x)=\frac{-R_g(x)}{6}$ (see e.g.~\cite[\S8.6]{Dang2020}), setting  $k=\frac{n}{2}-1$ and taking the limit $\varepsilon\to 0^+$, we find the  relation \eqref{eq:main2} between the dynamical residue and the Einstein--Hilbert action stated  in the introduction.

\appendix

 \section{Lorentzian canonical trace density}\label{app}\init 
 
 \subsection{Summary} A classical result due to Kontsevich--Vishik \cite{Kontsevich1995} says that if $A\geqslant 0$ is an elliptic operator on a compact manifold $M$  and $Q$ is (for instance) a differential operator, then the trace of $Q A^{-\cv}$ exists for large $\Re \cv$ and analytically continues to $\cc\setminus \zz$. In greater generality, the same is true  for the trace density, defined by on-diagonal restriction of the Schwartz kernel. The analytic continuation is called the \emph{Kontsevich--Vishik canonical trace density} and it plays a fundamental rôle in the definition of \emph{weighted traces} or \emph{renormalized traces}, see e.g.~\cite{paycha} and references therein.
 
 A very natural question\footnote{This was kindly suggested to us by an anonymous referee, whom we would like to thank heartily.} is whether elliptic complex powers $A^{-\cv}$ can be replaced by Lorentzian complex powers $(P-i\varepsilon)^{-\cv}$ in the setting of the wave operator $P=\square_g$ introduced  in \sec{section4}. In this appendix we provide an affirmative answer.

\subsection{Lorentzian canonical trace}   We prove the following result, assuming for the sake of simplicity that $Q$ is a differential operator. We leave for further studies the  case when $Q$ is a properly supported pseudodifferential with polyhomogeneous symbol of integer order.

\begin{thm}
 Let $(M,g)$ be a time-oriented Lorentzian manifold of even dimension $n$, and suppose $\square_g$ has Feynman resolvent $(P-z)^{-1}$. For any $\cv\in\cc\setminus \mathbb{Z}$ and $\Im z > 0$ and for any differential operator $Q$ of degree $q$, denote by $K_\cv$ the Schwartz kernel of $Q(P-z)^{-\cv}$. Then the on-diagonal restriction
\begin{equation}
\iota_\Delta^*\left(K_\cv \right)\in C^\infty(M)
\end{equation}  
is well-defined for $\Re \cv$ large enough and analytically continues to $\cv\in \mathbb{C}\setminus \mathbb{Z}$.
\end{thm}

\begin{proof}
We start from the decomposition
$(P-z)^{-\cv}=H_N^{(\cv)}(z)+R_N^{(\cv)}(z)$ for $N$ large enough so that the remainder term $R_N^{(\cv)}(z)$ belongs to  $\pazocal{C}^{s}_{\loc}$ for $s>q$. Then,  $QR_N^{(\cv)}(z)$ has a continuous Schwartz kernel which has a well-defined restriction to the diagonal.

We use the asymptotic expansion 
$$
H_N^{(\cv)}(z,.) = 
 \sum_{\varm=0}^N  u_\varm(.)\frac{(-1)^\varm\Gamma(-\cv+1)}{\Gamma(-\cv-\varm+1)\Gamma(\cv+\varm)} \mathbf{F}_{\varm+\cv-1}(z,.).
 $$
We study $QH_N^{(\cv)}(z,.)$, which can be expressed as a finite sum of smooth functions (these have a well-defined on-diagonal restriction) 
times derivatives of distributions of the form
$\partial_x^{\beta_1}\partial_h^{\beta_2} \mathbf{F}_{\varm+\cv-1}(z,.)$ where $\vert\beta_1\vert+\vert\beta_2\vert\leqslant q$.
Without loss of generality we can reduce the problem to the case when
$Q=\partial_x^{\beta_1}\partial_h^{\beta_2}$ in local coordinates $(x,h)$.

We use the notation from  Lemma~\ref{l:kuranishi}. 
We start again from the oscillatory integral representation
$$\mathbf{F}_\cv(z,x,h)=\frac{\Gamma(\cv+1)}{(2\pi)^n}\int_{\mathbb{R}^n} e^{i\left\langle \xi,h\right\rangle} \left(Q((^tM(x,h))^{-1}\xi  )-z \right)^{-\cv-1}\vert M(x,h) \vert^{-1}d^n\xi . $$
Let $\psi\in C^\infty_\c(\mathbb{R}^n)$ with $\psi=1$ near $0$. Then we use $\psi$ as a frequency cutoff.
We  expand the integrand in both variables $z$ and then smoothly in the parameters $(x,h)$. Namely,
$$ 
\bea
&\left(Q((^tM(x,h))^{-1}\xi  )-z \right)^{-\cv-1}(1-\psi)(\xi)\\
&=(1-\psi)(\xi)\sum_{p=0}^N (-1)^pz^p \frac{\Gamma(-\cv)}{\Gamma(p+1)\Gamma(-\cv-p)} \left(Q((^tM(x,h))^{-1}\xi  )-i0 \right)^{-\cv-1-p} + \text{remainder}
\eea
$$
where the omitted remainder terms  are weakly homogeneous of 
degree $\geqslant -\Re(\cv)-1-N$ in $\xi$, hence by inverse Fourier transform
they  have high H\"older regularity if $N$ is chosen large enough so that the inverse Fourier can be restricted to the diagonal. 
Then the second step is to Taylor expand the distribution 
$\left(Q((^tM(x,h))^{-1}\xi  )-i0 \right)^{-\cv-1-p}(1-\psi)(\xi)$ in the variable $h$. 
For all $\cv$,
$$
\bea
 &(Q(\t M(x, h)^{-1}\xi)-i0)^{-\cv}(1-\psi)(\xi) \\ &=(1-\psi)(\xi)\sum_{\ell, \vert \beta_1\vert+\dots+\vert\beta_\ell\vert\leqslant N} h^\beta Q_\beta( x,h;\xi)
|_{(x,0)}+I_N(z,x,h;\xi).
\eea
$$
where we denoted
$$
\bea
Q_\beta( x,h;\xi) & = \frac{(-\cv)\dots(-\cv-\ell-1) \left(\partial_h^{\beta_1}Q(\t M^{-1}(x,h)\xi)\right)\dots
\left(\partial_h^{\beta_\ell}Q(\t M^{-1}(x,h)\xi)\right) }{\beta_1!\dots\beta_\ell!\ell!} 
 \fantom  \times (Q(\xi)-i0)^{-\cv-\ell}.
\eea
$$

Note that in the present situation we do not need to take the finite part since $\Re\cv$ is large enough and we have the $1-\psi$ cutoff, which vanishes near $\xi=0$.
Each  $h^\beta Q_\beta( x,h;\xi) |_{(x,0)}$
term is \emph{polynomial} in $h$ and is a distribution in $\xi$, homogeneous of degree $-2\cv$, of order $\plancher{\Re\cv}+\ell+1$.
The integral remainder
$I_N(z,x,h;\xi)$ is continuous in $(x,h)$ with values in distributions in $\xi$, homogeneous of degree $-2\cv$,
of order $\plancher{\Re\cv}+N+2$ uniformly in $(x,h)$.

From now on the analysis in \secs{ss:polhom}{ss:rcc} can be applied with merely minor changes. For the sake of brevity we write `$\sim$' to denote identities which hold true modulo remainders as those discussed in \secs{ss:polhom}{ss:rcc}, which are H\"older regular enough to be restricted on the diagonal. In particular, we can write
$$
H_N^{(\cv)}(z) \sim \sum_{k=0}^N u_k \frac{\cv\dots(\cv+k-1)}{(2\pi)^n} \int_{\mathbb{R}^n} (1-\psi) e^{i\left\langle\xi,h \right\rangle}
(Q(\t M^{-1}(x,h)\xi)-i0)^{-k-\cv} \module{M(x,h)}^{-1} d^n\xi.
$$
Expanding in $z$ yields
$$ 
\bea
H_N^{(\cv)}(z)&\sim \sum_{k=0}^N\sum_{p=0}^\infty u_k  (-1)^pz^p  \begin{pmatrix}
-k-\cv
\\
p
\end{pmatrix}   \frac{\cv\dots(\cv+k-1)}{(2\pi)^n} \fantom \times  \int_{\mathbb{R}^n}(1-\psi) e^{i\left\langle\xi,h \right\rangle}
(Q(\t M^{-1}(x,h)\xi)-i0)^{-k-\cv-p} \module{M(x,h)}^{-1} d^n\xi
\\ & \sim  \sum_{  k=0}^\infty \sum_{p=0}^\infty z^p  u_k \frac{\cv\dots(\cv+k+p-1)}{p!(2\pi)^n} \fantom\times\int_{\mathbb{R}^n} (1-\psi)e^{i\left\langle\xi,h \right\rangle} 
(Q(\t M^{-1}(x,h)\xi)-i0)^{-k-\cv-p} \module{M(x,h)}^{-1} d^n\xi.
\\ & \sim  \sum_{  k=0}^\infty \sum_{p=0}^\infty z^p \module{M(x,h)}^{-1} u_k \frac{\cv\dots(\cv+k+p-1)h^\beta}{p!(2\pi)^n\beta!} \fantom\times\int_{\mathbb{R}^n}(1-\psi) e^{i\left\langle\xi,h \right\rangle} 
\partial_h^\beta(Q(\t M^{-1}(x,h)\xi)-i0)^{-k-\cv-p}|_{(x,0)}  d^n\xi.
\eea
$$
Above, the omitted remainders are H\"older regular enough to be restricted on the diagonal. Indeed, we can truncate the above series to some finite sum if we expand in $k+p$ large enough since $\partial_h^\beta(Q(\t M^{-1}(x,h)\xi)-i0)^{-k-\cv-p}|_{(x,0)} $ is a tempered distribution homogeneous of degree $-k-\cv-p $ in $\xi$, and therefore the inverse Fourier transform is sufficiently H\"older regular in $(x,h)$ and can be restricted to the diagonal. 
Note that we do not need to use finite parts anymore since all distributions are homogeneous on the support of $1-\psi$, which avoids $\xi=0$.
%
%
When we differentiate each term
$$z^p \module{M(x,h)}^{-1} u_k \frac{\cv\dots(\cv+k+p-1)h^\beta}{p!(2\pi)^n\beta!} \times\int_{\mathbb{R}^n} e^{i\left\langle\xi,h \right\rangle} 
\partial_h^\beta(Q(\t M^{-1}(x,h)\xi)-i0)^{-k-\cv-p}|_{(x,0)}  d^n\xi$$
on the r.h.s.~of the previous equality with the operator $\partial_x^{\beta_1}\partial_h^{\beta_2}$,
we get some finite combinations of terms of the form $$
\int_{\mathbb{R}^n}  \xi^\delta
\partial_h^\beta(Q(\t M^{-1}(x,h)\xi)-i0)^{-k-\cv-p}|_{(x,0)}  d^n\xi
\times (\text{smooth function}\in C^\infty(M\times M)).
$$
Each term $\partial_h^\beta(Q(\t M^{-1}(x,h)\xi)-i0)^{-k-\cv-p}|_{(x,0)} $ reads as the sum:
$$ 
\bea
  &\beta! \sum_{\ell,\beta_1+\dots+\beta_\ell=\beta} \frac{(-k-\cv-p)! \left(\partial_h^{\beta_1}Q(\t M^{-1}(x,h)\xi)\right)\dots
\left(\partial_h^{\beta_\ell}Q(\t M^{-1}(x,h)\xi)\right) }{(-k-\cv-p-\ell-1)!\beta_1!\dots\beta_\ell!\ell!} |_{h=0}
 \fantom  \times (Q(\xi)-i0)^{-k-\cv-p-\ell}=\sum_\ell \text{smooth function }\times (Q(\xi)-i0)^{-k-\cv-p-\ell}
\eea
$$
where the term is homogeneous of degree $-k-\cv-p$ in the $\xi$ variable. 
 Then we are reduced to prove that 
the general term
$$\int_{\mathbb{R}^n}(1-\psi)\xi^\delta (Q(\xi)-i0)^{-k-\cv-p-\ell}d^n\xi , \quad  \vert\delta\vert\leqslant q+2\ell, $$
for $\delta$ a multi-index,
is well-defined for $\Re \cv$ large enough and
has an {analytic continuation} to $\cv\in \mathbb{C}\setminus \mathbb{Z}$. The idea is again to use a Littlewood--Paley decomposition in momentum
$$1=\psi+\sum_{j=1}^\infty \beta(2^{-j}.) $$ and the homogeneity of the distribution:
$$
\bea
&\int_{\mathbb{R}^n}(1-\psi)\xi^\delta (Q(\xi)-i0)^{-\cv-k-p-\ell} d^n\xi \\
&=\sum_{j=1}^\infty \int_{\mathbb{R}^n}\xi^\delta \beta(2^{-j}\xi) (Q(\xi)-i0)^{-\cv-k-p-\ell} d^n\xi\\
&=\sum_{j=1}^\infty 2^{j(n+\vert\delta\vert)}\int_{\mathbb{R}^n}\xi^\delta \beta(\xi) (Q(2^j\xi)-i0)^{-\cv-k-p-\ell} d^n\xi\\
&=2(1-2^{n+\vert\delta\vert-2(\cv+k+p+\ell)})^{-1}\int_{\mathbb{R}^n}\xi^\delta \beta(\xi) (Q(\xi)-i0)^{-\cv-k-p-\ell} d^n\xi,
\eea
$$
where the last term admits a unique holomorphic continuation to $\cv\in \mathbb{C}\setminus \{ \mathbb{Z}\cap \,]-\infty, \frac{n}{2} +q] \}$, where we used the inequality $\vert\delta\vert\leqslant q+2\ell$.
This concludes the proof.
\end{proof}

\bibliographystyle{abbrv}
\bibliography{complexpowers}

\end{document}